\documentclass[final]{amsart}

\usepackage{amssymb}
\usepackage{amsaddr}\usepackage{amsaddr}

\usepackage[utf8]{inputenc}
\usepackage[T1]{fontenc}
\usepackage[english]{babel}
\usepackage{xargs}
\usepackage[all]{xy}

\usepackage{graphicx, color}
\setkeys{Gin}{draft=false}
\usepackage[figurewithin=none]{caption,subcaption} \graphicspath{{./}} \usepackage{enumitem}
\setlist[enumerate,1]{label={(\roman*)}}

\usepackage{tabularx}
\newcolumntype{C}{>{\centering\arraybackslash}X}\newcolumntype{R}{>{\raggedleft\arraybackslash}X}\newcolumntype{L}{>{\raggedright\arraybackslash}X}

\newcolumntype{\ll}[1]{>{\hsize=#1\hsize\raggedright\arraybackslash}X}\newcolumntype{\rr}[1]{>{\hsize=#1\hsize\raggedleft\arraybackslash}X}\newcolumntype{\cc}[1]{>{\hsize=#1\hsize\centering\arraybackslash}X}

\makeatletter
    \providecommand\@dotsep{5}
\makeatother

\usepackage[notcite,notref]{showkeys} \usepackage{hyperref}
\hypersetup{
  draft=false,
	hidelinks,
  colorlinks  = true,    urlcolor    = blue,    linkcolor   = blue,    citecolor   = blue,      pdfauthor   = {Gabe Cunningham, Elías Mochán, Antonio Montero},pdfsubject  = {Extensions of maniplexes},pdftitle    = {Cayley extensions of maniplexes},  }

\theoremstyle{plain}
\newtheorem{thm}{Theorem}[section]
\newtheorem{theorem}[thm]{Theorem}

\newtheorem{lemma}[thm]{Lemma}

\newtheorem{proposition}[thm]{Proposition}

\newtheorem{corollary}[thm]{Corollary}

\theoremstyle{definition}

\newtheorem{definition}[thm]{Definition}

\newtheorem{example}[thm]{Example}

\theoremstyle{remark}

\newtheorem{remark}[thm]{Remark}

\numberwithin{equation}{section}

\usepackage{xargs}
\usepackage[mathscr]{euscript}

\renewcommand{\leq}{\leqslant} \renewcommand{\geq}{\geqslant}
\renewcommand{\epsilon}{\varepsilon} \renewcommand{\subset}{\subseteq}  
\renewcommand{\{}{\lbrace}
\renewcommand{\}}{\rbrace}
\newcommand{\sm}{\setminus}
\renewcommand{\bar}{\overline}
\newcommand{\ol}[1]{\overline{#1}}
\newcommand{\os}[1]{\widetilde{#1}}
\newcommand{\comment}[1]{}
\newcommand\restr[2]{{\left.\kern-\nulldelimiterspace #1 \vphantom{\big|} \right|_{#2} }}

\newcommand{\bD}{\mathbb{D}}

\newcommand{\bN}{\mathbb{N}}

\newcommand{\bZ}{\mathbb{Z}}

\newcommand*{\cal}{\mathcal}

\newcommand{\cK}{\mathcal{K}}
\newcommand{\calK}{\mathcal{K}}
\newcommand{\cM}{\mathcal{M}}
\newcommand{\calM}{\mathcal{M}}

\newcommand{\cP}{\mathcal{P}}
\newcommand{\calP}{\mathcal{P}}

\newcommand{\cU}{\mathcal{U}}
\newcommand{\calU}{\mathcal{U}}
\newcommand{\cX}{\mathcal{X}}
\newcommand{\cY}{\mathcal{Y}}

\newcommand{\vect}[1]{\bar{\mathrm{#1}}}
\newcommand{\vx}{\vect{x}}

\newcommand{\va}{\vect{a}}
\newcommand{\ve}{\vect{e}}

\newcommand{\rr}{\varrho}

\DeclareMathOperator{\aut}{\Gamma} \DeclareMathOperator{\Aut}{Aut} 
\DeclareMathOperator{\cay}{Cay}

\DeclareMathOperator{\fg}{\Pi}

\DeclareMathOperator{\Mon}{Mon}
\DeclareMathOperator{\mon}{\Mon}

\newcommand{\covers}{\searrow}

\newcommand{\gen}[1]{\langle #1 \rangle}
\newcommand{\frnd}{\heartsuit}
\newcommand{\Part}{\mathscr{P}}

\newcommandx{\dtwoSM}[2][1=\cM, 2=s] {\hat{2}#2^{#1 - 1}}

\newcommand{\noshow}[1]{}

\usepackage[sorting=nyt, backend=biber, style=numeric,defernumbers=true,url=false,eprint=false,giveninits=true,
maxbibnames=9,isbn=false,sortcites = true,citestyle=numeric-comp,
]{biblatex}
\makeatletter
\let\blx@rerun@biber\relax
\makeatother

\begin{document}

\title{Cayley extensions of maniplexes and polytopes}

\author{Gabe Cunningham}
\address{Wentworth Institute of Technology, 02115  Boston MA, USA}
\email{cunninghamg1@wit.edu}

\author{Elías Mochán}
\address{Department of Mathematics, Northeastern University, 02115 Boston MA, USA}
\email{j.mochanquesnel@northeastern.edu}

\author{Antonio Montero}
\address{Faculty of Mathematics and Physics, University of Ljubljana, 1000 Ljubljana, Slovenia}
\email{antonio.montero@fmf.uni-lj.si}

\keywords{Maniplexes, maps, abstract polytopes}

\subjclass[2020]{Primary: 52B15 Secondary:  05C25, 05E18.}

\begin{abstract}
A map on a surface whose automorphism group has a subgroup acting regularly on its vertices is called a Cayley map. Here we generalize that notion to maniplexes and polytopes. We define $\calM$ to be a \emph{Cayley extension} of $\calK$ if the facets of $\calM$ are isomorphic to $\calK$ and if some subgroup of the automorphism group of $\calM$ acts regularly on the facets of $\calM$. We show that many natural extensions in the literature on maniplexes and polytopes are in fact Cayley extensions. We also describe several universal Cayley extensions. Finally, we examine the automorphism group and symmetry type graph of Cayley extensions.
\end{abstract}

\maketitle

\section{Introduction}

    Maniplexes, introduced in \cite{Wilson_2012_ManiplexesPart1}, are combinatorial structures that generalize maps on surfaces and abstract polytopes.
    Indeed, maniplexes of rank $3$ are maps (and vice-versa), and every maniplex $\cM$ of rank $n+1$ can be constructed by gluing together maniplexes of rank $n$ (the facets of $\cM$) in an appropriate way.
    Whenever all the facets of $\cM$ are isomorphic to a given maniplex $\cK$, then $\cM$ is an \emph{extension} of $\cK$.
    Many tools and ideas from the study of maps can be adapted to maniplexes, but often there are additional difficulties to be overcome in higher ranks.

    In this paper, we study \emph{Cayley (co)extensions} of maniplexes, inspired by Cayley graphs and maps which are named after A. Cayley.
Given a group $\Gamma$ and a set $\Sigma$ such that $\sigma^{-1} \in \Sigma$ for every $\sigma \in \Sigma$, the \emph{Cayley graph} $\cay(\Gamma, \Sigma)$ is a graph whose vertices are the elements of $\Gamma$ and two vertices $\alpha,\beta \in \Gamma$ are connected if and only if $\alpha \beta^{-1} \in \Sigma$.
    A classical theorem by Sabidussi \cite{sabidussi1958class} proves that a graph $X$ is a Cayley graph if and only if there is a group $\Gamma \leq \Aut(X)$ acting regularly on the vertices of $X$.
    This notion has been generalized to maps \cite{richter2005cayley}: an orientable map $M$ is a \emph{Cayley map} if there exists a group of orientation-preserving automorphisms acting regularly on the vertices of $M$.
    There have been generalizations of Cayley maps to include maps on non-orientable surfaces (see \cite{KwakKwon_2006_UnorientedCayleyMaps}, for example) but the theory is not as developed as that of orientable Cayley maps.

    Analogously to Cayley maps, we define a maniplex $\calM$ to be a \emph{Cayley extension} (resp. coextension) if some subgroup of the automorphism group of $\calM$ acts regularly on the facets (resp. vertices) of $\calM$. This provides one of the first generalizations of the Cayley property to the context of maniplexes and abstract polytopes.
    Using voltage graphs, we describe how to create Cayley extensions of a given facet with a given group acting regularly on the facets. We will show that many natural extensions of maniplexes and polytopes can be constructed in this way.

    In Section\nobreakspace \ref {sec:prelim} we review the basics of maniplexes and voltage graphs. Then in Section\nobreakspace \ref {sec:extensions}, we define Cayley extensions formally, and we describe some basic properties. We consider the question of when a Cayley extension is an abstract polytope in Section\nobreakspace \ref {sec:polytopality}. Section\nobreakspace \ref {sec:universal} demonstrates several examples of known extensions that can be described using our framework, and then defines new classes of universal extensions. Finally, Section\nobreakspace \ref {sec:symmetries} determines the symmetry type graphs of these universal extensions, and characterizes when different ways of constructing a universal extension lead to the same result.
 \section{Preliminaries}\label{sec:prelim}

\subsection{Graphs}
In this work we use the definition of graph used in \cite{MalnicNedelaSkoviera_2000_LiftingGraphAutomorphisms}, which is slightly more general than the typical definition.
A \emph{graph} $X$ is a quadruple $(D,V,I,(\cdot)^{-1})$ where $D$ and $V$ are disjoint sets, $I:D \to V$ is a function and $(\cdot)^{-1}: D \to D$ is an involutory permutation of $D$.
The set $V$ is the set of \emph{vertices} of $X$, and the set $D$ is the set of \emph{darts} of $X$.
For a dart $d$, the vertex $I(d)$ is the \emph{initial vertex} or \emph{starting point} of $d$ and $d^{-1}$ is the \emph{inverse} of $d$. The \emph{terminal vertex} or \emph{endpoint} of $d$ is the starting point of $d^{-1}$.
The \emph{edges} are the orbits of $D$ under the action of $(\cdot)^{-1}$. If an edge consists only of a single dart, that is, a dart $d$ that satisfies $d^{-1}=d$, then it is called a \emph{semiedge}.
A \emph{loop} is an edge consisting of two darts whose initial vertex is the same.
A \emph{link} is an edge that is not a loop or a semiedge;
that is, an edge whose two darts are different and have different starting points.

Most classical graph theoretical concepts extend naturally to our definition; we shall just make precise a few of them that are relevant for this manuscript.

If $X=(V(X),D(X), I_{X}, (\cdot)^{-1}_{X})$ and $Y=(V(Y),D(Y), I_{Y}, (\cdot)^{-1}_{Y})$ are graphs, a \emph{graph homomorphism} $f: X \to Y$ is a pair of functions $f=(f_V, f_D)$ such that $f_V: V(X) \to V(Y)$, $f_D: D(X) \to D(Y)$, and
\[\begin{aligned}
I_{Y}\left( d f_{D} \right) &= \left( I_{X} (d) \right)f_V && \text{and} \\
\left( d f_D \right)_{Y}^{-1} &= ( d^{-1}_{X} )f_D
\end{aligned}
\]
for every dart $d \in D(X)$.
Note that we evaluate graph homomorphisms on the right.

If $v\in V(X)$ and $d \in D(X)$, we shall write $vf$ and $df$ instead of $vf_{V}$ and $df_{D}$.
If both $f_{V}$ and $f_{D}$ are bijective and $f^{-1}:=(f_{V}^{-1},f_{D}^{-1})$ is also a graph homomorphism, then we say that $f$ (and $f^{-1})$ is an \emph{isomorphism} and that $X$ and $Y$ are \emph{isomorphic} (and write $X \cong Y$).
Naturally, for a graph $X$, an \emph{automorphism} is an isomorphism $f:X \to X$.

A \emph{path} is a finite sequence  $W=(d_{1}, \dots, d_{k})$ of darts such that the endpoint of $d_{i}$ is the starting point of $d_{i+1}$ for $i \in \{0, \dots, k-1\}$.
We usually omit commas and parentheses and simply write $W=d_{1} \cdots d_{k}$.
The number $k$ is the \emph{length} of $W$.
The starting point of $d_1$ is the \emph{starting point} of $W$; we also say that $W$ \emph{starts} at this vertex.
Similarly, the \emph{endpoint} of $W$ is the endpoint of $d_{k}$ and we say that $W$ \emph{ends} at this vertex.
A single vertex is a path of length $0$.
A path is \emph{closed} if its starting point and its endpoint are the same vertex.

An \emph{elementary graph move} on a path consists of removing or inserting two consecutive inverse darts.
Two paths $W_{1}$ and $W_{2}$ are \emph{graph-homotopic} if their initial vertex is the same and $W_{1}$ can be obtained from $W_{2}$ after a (possibly empty) sequence of elementary graph moves.
In this case we write $W_{1} \sim W_{2}$.
Clearly graph-homotopy is an equivalence relation and we often identify a path with its homotopy class.
Observe that if $W$ is graph-homotopic to a path of length $0$ then $W$ must be closed.

Given two paths $W_1$ and $W_2$ we say that they are \emph{compatible} if the endpoint of $W_1$ is the starting point of $W_2$.
We can operate on compatible paths by concatenation.
That is, if $W_1=d_{1} \cdots d_{k}$ and $W_2=a_{1} \cdots a_{\ell}$ then $W_1W_2= d_{1}  \cdots  d_{k}  a_{1}  \cdots  a_{\ell}$.
If $W_1 \sim W_1'$ and $W_2 \sim W_2'$ then $W_1W_2 \sim W_1'W_2'$, which implies that we can operate not only on compatible paths but also on homotopy classes of compatible paths.

The \emph{fundamental groupoid} of a graph $X$, denoted by $\fg(X)$, is the set of graph-homotopy classes of $X$ with the partial operation defined above.
If $u$ is a vertex in $X$, then the fundamental group of $X$ at $u$, denoted by $\fg^{u}(X)$, is the set of graph-homotopy classes of closed paths at $u$ with concatenation as operation.
We also denote by $\fg^{u,v}(X)$ the set of homotopy classes of paths from $u$ to $v$.
Observe that $\fg^{u}(X)$ is actually a group and that if $W \in \fg^{u,v}$ then  $\fg^{u}(X) = W \fg^{v}(X) W^{-1} $.

\subsection{Maniplexes and premaniplexes}

A \emph{properly $n$-edge-colored graph} is an $n$-valent graph $X$ with a coloring $c:D(X) \to \left\{ 0, \dots, n-1 \right\} $ such that $c(d_{1}) \neq c(d_{2})$ whenever $I(d_{1}) = I(d_{2})$, and $c(d) = c(d^{-1})$ for every dart $d$.
Observe that thanks to this last condition $c$ is well defined on edges.
If $d$ is a dart (resp. an edge) we say that $a$ is an $i$-dart (resp. $i$-edge) if $c(d)=i$.
Notice that each vertex $x$ of $X$ has a unique $i$-dart $d_{i}$ such that $I(d_{i}) = x$.
The endpoint of $d_{i}$ is the \emph{$i$-adjacent} vertex of $x$ and it is denoted by $x^{i}$.
Observe that a properly $n$-edge-colored graph cannot have loops but it might have semi-edges or parallel edges.
For a sequence $w=(i_{1}, \dots, i_{k})$ with $i_{1}, \dots, i_{k} \in \left\{ 0, \dots, n-1 \right\} $ a $w$-path is a path $(d_{1}, \dots, d_{k})$ such that $c(d_{j}) = i_{j}$ for $j \in \left\{ 1, \dots, k \right\} $.
In this situation we say that $w$ is the \emph{color sequence} of the path.
Observe that a path is completely determined by its initial vertex and its color sequence.

A \emph{premaniplex of rank $n$} or \emph{$n$-premaniplex} $\cX$ is a properly $n$-edge-colored graph such that every $(i,j,i,j)$-path is closed whenever $i$ and $j$ are non-consecutive.
If $\cX$ is also connected and the $(i,j,i,j)$-paths are not only closed but actually alternating $4$-cycles then $\cX$ is a \emph{maniplex}.

Natural examples of maniplexes are the flag-graphs of polytopes.
In fact, maniplexes were introduced by Wilson \cite{Wilson_2012_ManiplexesPart1} as a common generalization of maps and abstract polytopes.
The flag-graph of a map or an $n$-polytope is a $3$-maniplex or an $n$-maniplex respectively\footnote{In \cite{Wilson_2012_ManiplexesPart1}, the indexing is is shifted by 1, so that the flag-graph of a map is called a $2$-maniplex, for example. However, by now the established convention is the one we present here.}.
On the other hand, in \cite{GarzaVargasHubard_2018_PolytopalityManiplexes} Hubard and Garza-Vargas give a characterization of those maniplexes that are flag-graphs of polytopes.
The notion of  premaniplexes is slightly less natural but it conveniently arises when considering the symmetry type graphs of abstract polytopes.
For a polytope $\cP$ and $\Gamma \leq \aut(\cP)$, the \emph{symmetry type graph} of $\cP$ with respect to $\Gamma$ is the edge-colored graph obtained as the quotient of the flag graph of $\cP$ by $\Gamma$.
When $\Gamma = \aut(\cP)$ we just call it the \emph{symmetry type graph of $\cP$} (see \cite{CunninghamDelRioFrancosHubardToledo_2015_SymmetryTypeGraphs} for details).
Premaniplexes are the natural candidates to be symmetry type graphs of abstract polytopes, in the sense that every symmetry type graph is a premaniplex.
Connected premaniplexes are precisely what were called \emph{admissible graphs} in  \cite{CunninghamPellicer_2018_OpenProblems$k$}.

Maniplexes are often thought of as a graph theoretical version of abstract polytopes.
To emphasise this fact, we shall call the vertices of a maniplex \emph{flags}.
If $\cM$ is an $n$-maniplex and $i \in \{0, \dots, n-1\}$, the \emph{$i$-faces} of $\cM$ are the connected components of $\cM$ after removing the edges of color $i$.
The $(n-1)$-faces are called \emph{facets}.
The set of $i$-faces of $\cM$ is denoted by $(\cM)_{i}$ in particular, $(\cM)_{n-1}$ denotes the set of facets of $\cM$.
More generally, if $0 \leq k < \ell \leq n-1$, the \emph{$(k,\ell)$-sections} of $\cM$ are the connected components of $\cM$ after removing the edges of color $i$ for $i < k$ and $i > \ell$.
If $\Phi$ is a flag of $\cM$ and $i \in \left\{ 0, \dots, n-1 \right\} $ we denote by $\Phi_{i}$ the $i$-face of $\Phi$.
The $(1,n-1)$-sections are called \emph{vertex-figures}.
The notions of faces, facets, sections and vertex-figures extend naturally to premaniplexes.

The \emph{universal string Coxeter group of rank $n$} is the group $W^{n}$ generated by $r_{0}, \dots, r_{n-1}$ subject to the following defining relations:
\[\begin{aligned}
	r_{i}^{2} &= 1 && \text{for } i \in \left\{ 0, \dots, n-1 \right\}, \\
	(r_{i}r_{j})^{2} &=1 && \text{if } \left| i-j \right| > 1.
\end{aligned}\]
If $\cX$ is an $n$-premaniplex the group $W^{n}$ acts on the vertices of $\cX$ by \[r_{i} x = x^{i}.\]
The induced permutation group $\mon(\cX)$ is the \emph{monodromy group} (also called the \emph{connection group}) of $\cX$ and the elements of $\mon(\cX)$ are called \emph{monodromies}.
Note that we evaluate monodromies on the left.
Observe that a path with color sequence $(i_{1}, \dots, i_{k})$ determines the element $w=r_{i_{k}} \cdots r_{i_{1}} \in W^{n}$ such that if the path starts at $x$ then the endpoint is $wx$.

The group $\mon(\cX)$ acts transitively on every connected component of $\cX$, hence $\cX$ is connected if and only if $\mon(\cX)$ is transitive and in this case it is a maniplex if and only if the permutations induced by $r_{i}$ and by $r_{i}r_{j}$ for $i,j \in \left\{0, \dots, n-1  \right\} $ are fixed-point-free.

Given two $n$-premaniplexes $\cX$ and $\cY$, a \emph{homomorphism} $\phi: \cX \to \cY$ of premaniplexes is a color-preserving graph homomorphism.
This is equivalent to \[(r_{i}x) \phi = r_{i}(x \phi)\] for every $i \in \left\{ 0, \dots, n-1 \right\} $.
A homomorphism $\phi: \cX \to \cY$ is a \emph{covering} if for every vertex of $y $ of $ \cY$ there exists a vertex $x$ in $\cX$ such that $x \phi = y$.
Almost every homomorphism is a covering in the sense that if $y$ is in the image of $\phi$, then the whole connected component of $\cY$ containing $y$ is in the image of $\phi$.
More precisely, if $y_{1}$ and $y_{2}$ lie in the same connected component of $\cY$, then there exists $w \in W^{n}$ such that $y_{2} = w y_{1}$.
Hence if $x_{1} \phi = y_{1}$, then
\[(wx_{1}) \phi = w (x_{1} \phi) = w y_{1} = y_{2}. \]
It follows that every homomorphism is determined by the image of a given vertex of each connected component of $\cX$, and if $\cY$ is connected, every homomorphism $\phi: \cX \to \cY$ is a covering.
This notion of coverings has been previously treated in the context of abstract polytopes where they are often called \emph{rap-maps} (see \cite[Section 2]{MonsonPellicerWilliams_2014_MixingMonodromyAbstract} and \cite[Section 2D]{McMullenSchulte_2002_AbstractRegularPolytopes})

The notions of isomorphism and automorphism of premaniplexes extend naturally.
If $\cX$ and $\cY$  are $n$-premaniplexes an \emph{isomorphism} $\phi: \cX \to \cY$ is just a color-preserving isomorphism and an automorphism of $\cX$ is an isomorphism $\phi: \cX \to \cX$.

Automorphisms of (pre)maniplexes are the graph theoretical analogue of automorphisms of abstract polytopes, which in turn generalize geometrical symmetries of convex polytopes and similar structures.
We denote by $\aut(\cX)$ the group of automorphisms of $\cX$.
It can be easily seen that if $\cX$ is connected then the action of $\aut(\cX)$ on $\cX$ is free.

If the action of $\aut(\cX)$ is also transitive we say that $\cX$ is \emph{regular}.
We point out that the word \emph{reflexible} has been used for flag transitive maniplexes; we use \emph{regular} as an attempt to unify the terminology with that of abstract polytopes.

Regular abstract polytopes are by far the most studied class of highly symmetric polytopes.
Many of the classical results in the theory of abstract regular polytopes extend naturally to regular maniplexes.
In particular, if $\cM$ is a regular $n$-maniplex and $\Phi$ is a fixed base flag, then for each $i \in \left\{ 0, \dots, n-1 \right\} $ there exists a (unique) automorphism $\rho_{i}$ such that \[\Phi \rho_{i} = r_{i} \Phi.\]
It can be proved easily that $\aut(\cM) = \left\langle \rho_{0}, \dots, \rho_{n-1} \right\rangle $.

A maniplex $\cM$ is \emph{chiral} if $\aut(\cM)$ has two orbits on flags in such a way that adjacent flags lie in different orbits.
Chirality is the combinatorial version of a geometric polytope-like object admitting full rotational symmetry but no reflections.
Chiral abstract polytopes were introduced in the 90's by Schulte and Weiss \cite{SchulteWeiss_1991_ChiralPolytopes} and since then they have been extensively studied. See \cite{Pellicer_2012_DevelopmentsOpenProblems} for a nice review of the basic theory and interesting problems on the theory of chiral polytopes.

A natural consequence of $\aut(\cM)$ acting freely on $\cM$ is that we can use the number of flag-orbits of $\aut(\cM)$ as a degree of symmetry of a maniplex.
A maniplex $\cM$ is a \emph{$k$-orbit} maniplex if $\aut(\cM)$ has $k$ flag-orbits.
Chiral maniplexes are just one of $2^{n-1}$ symmetry types of $2$-orbit maniplexes.
The mere existence of maniplexes for all such symmetry types was recently established \cite{PellicerPotocnikToledo_2019_ExistenceResultTwo}.
The theory of $k$-orbit polytopes  for $k\geq 3$ is even less explored.
The possible symmetry types for $3$- and $4$-orbit maniplexes were described in \cite{CunninghamDelRioFrancosHubardToledo_2015_SymmetryTypeGraphs}, where the corresponding automorphism groups were also described.
Very recently, Hubard and the second author established the necessary conditions to build $k$-orbit abstract polytopes as coset geometries (see \cite{HubardMochan_AllPolytopesAre_preprint}).
See \cite{CunninghamPellicer_2018_OpenProblems$k$} for a review of the basic theory and open problems for $k$-orbit polytopes.

If $\cX$ is an $n$-premaniplex, the dual of $\cX$ (often denoted by $\cX^{\ast}$) is the premaniplex with the same vertex set and such that two vertices are $i$-adjacent in $\cX^{\ast}$ if and only if they are $(n-i-1)$-adjacent in $\cX$.
Notice that up to a permutation of the color set, if $F$ is a facet in $\cX$, then the graph induced by the vertices of $F$ in $\cX^{\ast}$ is isomorphic to $F^{\ast}$ and it is a vertex-figure of $\cX^{\ast}$.

The key notion of this paper is that of \emph{extensions}.
If $\cM$ is an $(n+1)$-maniplex such that all its facets are isomorphic to a given $n$-maniplex $\cK$, then we say that $\cM$ is an extension of $\cK$.
Similarly, if all the vertex figures of $\cM$ are isomorphic to $\cK$ we say that $\cM$ is a \emph{coextension} of $\cK$.
Observe that $\cM$ is a coextension of $\cK$ if and only if $\cM^{\ast}$ is an extension of $\cK^{\ast}$.
Extensions and coextensions have been deeply studied on the context of regular and chiral polytopes (see \cite{Cunningham_2021_FlatExtensionsAbstract,CunninghamPellicer_2014_ChiralExtensionsChiral,Montero_2021_SchlaefliSymbolChiral,Pellicer_2010_ExtensionsDuallyBipartite,Schulte_1985_ExtensionsRegularComplexes,SchulteWeiss_1995_FreeExtensionsChiral}, for example).

\subsection{Voltage assignments}
If $X$ is a graph and $G$ is a group, a \emph{voltage assignment} is a function $\xi: \fg(X) \to G$ that satisfies $\xi(W_1 W_2) = \xi(W_2)\xi(W_1)$  for every pair of compatible paths $W_1$ and $W_2$.
Notice that traditionally a voltage assignment $\xi$ is defined such that $\xi(W_1 W_2)=\xi(W_1)\xi(W_2)$ (cf. \cite{MalnicNedelaSkoviera_2000_LiftingGraphAutomorphisms}). The reason we do it the other way is because we are considering automorphisms acting on the right and monodromies on the left, as is customary in the polytopes and maniplexes literature.
The group $G$ is called the \emph{voltage group} of $\xi$ and the pair $(X,\xi)$ is called a \emph{voltage graph}.

Since every path can be thought as the product of its darts, we can regard a voltage assignment as a function $\xi:D(X) \to G$ such that for every dart $d$, $\xi(d^{-1})= \xi(d)^{-1}$. The voltage of a path $W= d_{1}, \dots, d_{k}$ is simply $\xi(d_{k}) \cdots \xi(d_{1})$.

If $(X,\xi)$ is a voltage graph, the \emph{derived graph} is the graph $X^{\xi}$ whose vertices and darts are the elements in $V(X) \times G$ and $D(X) \times G$, respectively.
The initial vertex of a dart $(d,g)$ is $(I(d),g)$ and its inverse is $(d^{-1}, \xi(d)g)$.
The adjacent vertices of a given vertex $(x,g)$ are the vertices $(y, \xi(d)g)$ for each dart $d$ starting at $x$ and ending at $y$.
Equivalently, two vertices $(x,g)$ and $(y, h)$ are connected if there exists a dart $d$ starting at $x$ and ending at $y$ whose voltage is $h g^{-1}$.

An \emph{elementary (maniplex) move} on a path $W$ consists of either adding or removing the same color two times at any two consecutive positions
or swapping two non-consecutive colors in consecutive positions in the color sequence.
More precisely, let $W_{0}$ be a path with color sequence $(i_{1}, \dots, i_{k})$ and starting vertex $x$.
Assume that  $i \in \left\{ 0, \dots, n-1 \right\} $ and that $j,\ell \in \left\{ 1, \dots, k \right\} $ with $\ell < k$ and $|i_\ell-i_{\ell+1}|>1$.
Let $W_{1}$ and $W_{2}$ be paths starting at $x$ and with color sequences $(i_{1}, \dots, i_{j}, i,i, i_{j+1}, \dots, i_{k})$ and $(i_{1}, \dots, i_{\ell - 1} , i_{\ell+1}, i_{\ell}, i_{\ell+2} \dots i_{k})$, respectively.
Then $W_{0} \mapsto W_{1}$, $W_{1} \mapsto W_{0}$
and $W_{0} \mapsto W_{2}$ are elementary maniplex moves.

We say that two paths in a premaniplex are \emph{maniplex-homotopic} if we can turn one into the other by a finite sequence of elementary maniplex moves. When it is clear from context, we will drop the modifier `maniplex' and merely call such paths homotopic.
Observe that two paths in a premaniplex are homotopic if and only if they start at the same vertex and determine the same element of $W^{n}$.
For an $n$-premaniplex $\cX$ and $x,y$ in $\cX$ we will abuse notation and denote by $\fg(\cX)$, $\fg^{x}(\cX)$ and $\fg^{x,y}(\cX)$ the fundamental groupoid, the fundamental group and the set of homotopy classes of paths starting at $x$ and ending at $y$ with respect to maniplex homotopy.
Moreover, if $I \subset \left\{ 0, \dots n-1\right\} $ we denote by $\fg_{I}(\cX)$, $\fg^{x}_{I}(\cX)$ and $\fg^{x,y}_{I}(\cX)$ the corresponding subsets of homotopy-classes containing paths using only colors in $I$.
When using this notation we often write $\bar{I}$ instead of $\left\{ 0, \dots, n-1 \right\} \sm I $ and if $I$ is an interval, meaning there exist $k \leq m$ such that $I = \left\{ i : k \leq i \leq m  \right\} $ we write $[k,m]$ instead of $I$.

Let $\cX$ be a premaniplex and let $\fg(\cX)$ be its fundamental groupoid.
If $\xi:\fg(\cX)\to G$ is a voltage assignment such that $\xi(W)$ is the identity in $G$ whenever $W$ is a path of length 4 alternating between two non-consecutive colors, we say that the pair $(\cX,\xi)$ is a \emph{voltage premaniplex}.
In other words, a voltage premaniplex is a voltage graph where the voltage of the maniplex homotopy class of a  path is well defined.

Our main technique in this paper is to build maniplexes as derived graphs of certain voltage premaniplexes.

 \section{Cayley extensions and coextensions}
\label{sec:extensions}

    \subsection{Definition and examples}\label{sec:defs_and_examples}

    Let $\calK$ be an $n$-maniplex and $\calM$ an extension of $\calK$.
Let $G \leq \Gamma(\calM)$.
    We say that $\calM$ is a \emph{Cayley extension of $\calK$ by $G$} if $G$ acts regularly (that is, transitively and freely) on the set of facets of $\calM$.
    If we say that \emph{$\calM$ is a Cayley extension of $\calK$} we mean that it is a Cayley extension by some group $G \leq \Gamma(\calM)$.

    Note that if $\calM$ is a Cayley extension of $\calK$, the quotient $\calK_{r_n} = \calM/G$ has only one facet, and that facet is isomorphic to $\calK$.
    In other words, $\calK_{r_n}$ consists of a copy of $\calK$ with some extra edges of color $n$, and satisfying that two darts of color $n$ starting at flags in a fixed facet $F$ of $\calK$ also end at some fixed facet $F'$.
    The monodromy $r_n$ in $\calK_{r_n}$ can be thought of as a set of isomorphisms between pairs of (not necessarily distinct) facets of $\calK$, in such a way that every facet is paired with exactly one other facet.

    If $\calM$ is a Cayley extension of $\calK$ we can recover $\calM$ as the derived graph of a voltage graph $(\calK_{r_n},\xi)$ where  $\xi:\Pi(\calK_{r_n})\to G$ satisfies the following properties:
    \begin{enumerate}
        \item The voltage $\xi(d)$ of a dart $d$ is trivial whenever the dart has a color smaller than $n$.
        \item If $d$ and $d'$ are two darts of color $n$ starting at flags in the same facet of $\calK$, then $\xi(d)=\xi(d')$.
        This means in particular that we may alternatively think of $\xi$ as a function from the set $(\calK)_{n-1}$ 
of facets of $\calK$ to the voltage group $G$.
    \end{enumerate}

    \begin{definition}
    A \emph{Cayley extender} is a triple $(\calK, r_n, \xi)$ where:
    \begin{enumerate}\item \label{item:maniplex} $\calK$ is an $n$-maniplex,
        \item \label{item:rn} $r_n$ is a permutation  of the flags of $\calK$ satisfying that $r_n^2=1$ and that if $F$ is a facet of $\calK$, then, the restriction $\restr{r_n}{F}$ is an isomorphism between $F$ and $r_n(F)$. (In other words, $r_n$ commutes with $\langle r_0, \ldots, r_{n-2} \rangle$),
        \item \label{item:voltageGroup} There is a group $G$, called the \emph{voltage group of $(\calK, r_n, \xi)$}, such that $\xi:(\calK)_{n-1}\to G$ is a function satisfying that $\xi(r_n(F)) = \xi(F)^{-1}$.
    \end{enumerate}
    Additionally, a pair $(\cK, r_{n})$ satisfying Items\nobreakspace \ref {item:maniplex} and\nobreakspace  \ref {item:rn} above is called a \emph{pre-extender} of $\cK$.
    \end{definition}

    Given a Cayley extender we can construct a voltage premaniplex $(\calK_{r_n},\xi)$
    as follows:
    the flags of $\calK_{r_n}$ are the flags of $\calK$ and for $i<n$ the $i$-adjacencies on $\calK_{r_n}$ are the same as in $\cK$.
    For a flag $\Phi$ in $\cK$, its  $n$-adjacent flag is $r_n \Phi$.
    Finally, the voltage $\xi(d)$ is $\xi(I(d)_{n-1})$, if $d$ has color $n$ and $1$ otherwise. That is, the voltage of a dart of color $n$ is the voltage of the facet at which it starts, and all other darts have trivial voltage.
The premaniplex $\calK_{r_n}$ is called a \emph{pre-extension} of $\cK$. 
    Note that the automorphism group $\Gamma(\calK_{r_n})$ consists of all the automorphisms of $\calK$ that commute with the monodromy $r_n$ and it may be a proper subgroup of $\Gamma(\cal K)$.
    
    Notice that the premaniplex $\calK_{r_n}$ is regular if and only if $\calK$ is regular and there exists $w \in W^{n}$ such that the $n$-adjacent flag of $\Phi$ is $w\Phi$ for every $\Phi$ in $\cK$.
    More generally if $\cK$ is a $k$-orbit maniplex with orbits $O_{1}, \dots, O_{k}$, then  $\cK_{r_{n}}$ has $k$ orbits (hence $\Gamma(\calK_{r_n})$ coincides with $\Gamma(\cal K)$) if and only if it is possible to chose $k$ elements $w_{1}, \dots, w_{k} \in W^{n}$ such that $\Phi^{n} = w_{i}\Phi$ whenever $\Phi \in O_{i}$.

    \begin{proposition} \label{prop:DerivedIsManiplex}
    The derived graph $\calK_{r_n}^\xi$ will be a maniplex if and only if:
    \begin{enumerate}
        \item The set $\xi(\calK_{n-1})=\{\xi(F) : F \in(\calK)_{n-1} \}$ is a generating set for $G$.
        \item If $\restr{r_n}{F}$ is the identity for some facet $F$, or if $r_n\Phi$ is adjacent to $\Phi$ for some flag $\Phi$ in $F$, then $\xi(F)\neq 1$.
    \end{enumerate}    
    \end{proposition}

    \begin{proof}
    The first condition ensures that the graph is connected, and the second condition ensures that $\calK_{r_n}^\xi$ has no semi-edges or parallel edges. 
    \end{proof}

    From now on, we will assume that our Cayley extenders satisfy the conditions of Proposition\nobreakspace \ref {prop:DerivedIsManiplex}.

    \begin{proposition}\label{prop:DerivedIsCayExt}
    A maniplex $\calM$ is a Cayley extension of $\calK$ by $G$ if and only if $\calM \cong \calK_{r_n}^\xi$ for some Cayley extender $(\calK, r_n, \xi)$ with voltage group $G$.
\end{proposition}

    \begin{proof}
First, suppose that $\calM$ is a Cayley extension of $\calK$ by $G$, where $\calK$ is an $n$-maniplex.
    In other words, $G$ acts regularly on the facets of $\calM$.
    Let $\calK$ be a particular facet of $\calM$.
    Then the quotient $\calM/G$ is isomorphic to a pre-maniplex $\calK_{r_{n}}$; a copy of $\calK$ with added edges of color $n$.
    For each facet $F$ of $\calK$ (which is a subfacet of $\calM$) let $\xi(F)$ be the automorphism in $G$ that maps $\calK$ to the facet that shares $F$ with it.
    Then $\calM$ is isomorphic to $\calK_{r_n}^\xi$.

    Conversely, suppose that $\calM \cong \calK_{r_n}^\xi$.
    Note that the facets of $\calM$ are of the form $(\calK,\gamma):=\{(\Phi,\gamma):\Phi\in\calK\}$ for some $\gamma \in G$, each of which is isomorphic to $\calK$. Note also that an element $\alpha\in G$ maps the facet $(\calK,\gamma)$ to the facet $(\calK,\gamma\alpha)$.
    From this it follows immediately that $G$ acts regularly on the facets of $\calK_{r_n}^\xi$, and so $\calM$ is a Cayley extension of $\calK$ by $G$.
    \end{proof}

    A more intuitive way of constructing $\calK_{r_n}^\xi$ is the following:
    Let $\{(\calK,\gamma) : \gamma \in G\}$ be a family of copies of $\calK$.
    These will be the facets of $\calK_{r_{n}}^{\xi}$.
    Then, for each facet $F$ of $\calK$ and for each $\gamma\in G$, glue the copy of $F$ in $(\calK,\gamma)$ to the copy of $r_n(F)$ in $(\calK,\xi(F)\gamma)$ using $r_n$ as gluing instructions.
    Let us consider some simple examples.

    \begin{example}
    Let $\calM$ be the unique $1$-maniplex, which has 2 flags. Let us color one flag white and the other one black. If we add semi-edges of color 1, and give them involutory voltages $\alpha_1$ and $\alpha_2$, then the derived graph is the polygon $\{2k\}$, where $k$ is the order of $\alpha_1 \alpha_2$ (Figure\nobreakspace \ref {fig:1ext_A}). If we instead add a dart of color 1, with voltage $\alpha$, then the derived graph is the polygon $\{k\}$, where $k$ is the order of $\alpha$ (Figure\nobreakspace \ref {fig:1ext_B}).
    \end{example}

   \begin{figure}
	\centering
	\begin{subfigure}[b]{\textwidth}
		\centering
			\begin{scriptsize}
\def\svgwidth{.6\textwidth}
		\begingroup \makeatletter \providecommand\color[2][]{\errmessage{(Inkscape) Color is used for the text in Inkscape, but the package 'color.sty' is not loaded}\renewcommand\color[2][]{}}\providecommand\transparent[1]{\errmessage{(Inkscape) Transparency is used (non-zero) for the text in Inkscape, but the package 'transparent.sty' is not loaded}\renewcommand\transparent[1]{}}\providecommand\rotatebox[2]{#2}\newcommand*\fsize{\dimexpr\f@size pt\relax}\newcommand*\lineheight[1]{\fontsize{\fsize}{#1\fsize}\selectfont}\ifx\svgwidth\undefined \setlength{\unitlength}{791.8527486bp}\ifx\svgscale\undefined \relax \else \setlength{\unitlength}{\unitlength * \real{\svgscale}}\fi \else \setlength{\unitlength}{\svgwidth}\fi \global\let\svgwidth\undefined \global\let\svgscale\undefined \makeatother \begin{picture}(1,0.37461091)\lineheight{1}\setlength\tabcolsep{0pt}\put(0,0){\includegraphics[width=\unitlength,page=1]{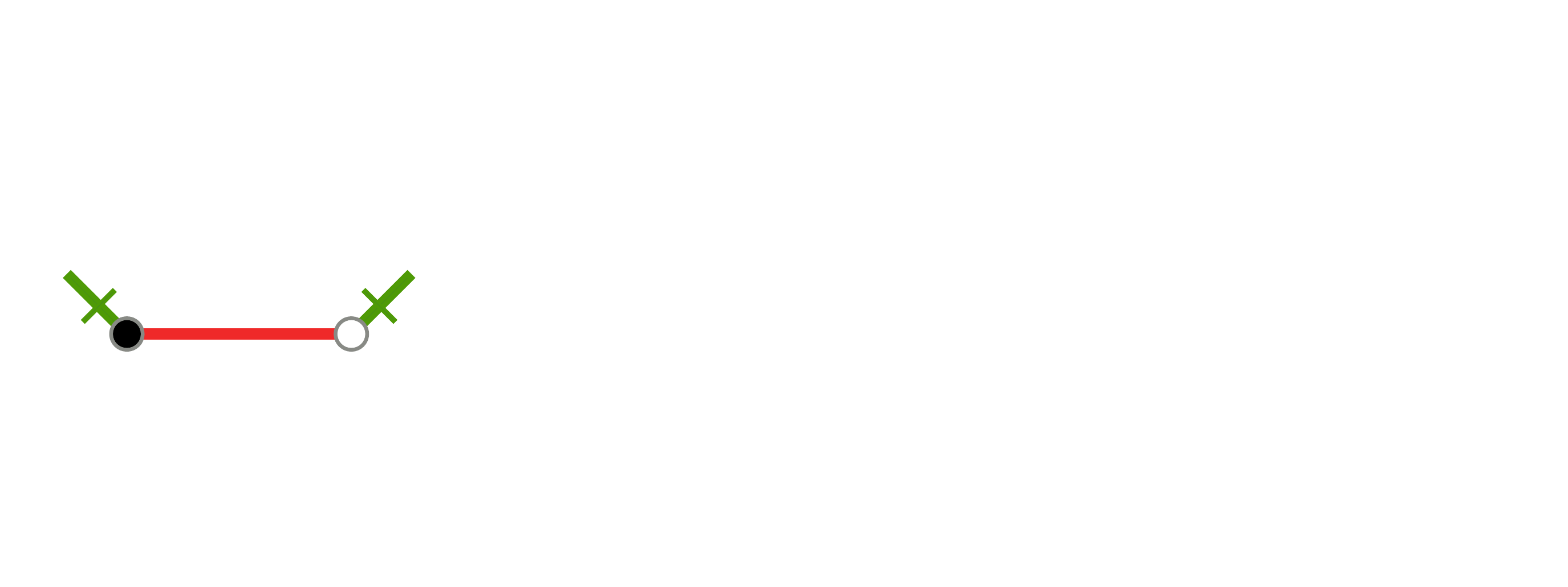}}\put(0.2538047,0.21742784){\color[rgb]{0.30588235,0.60392157,0.02352941}\makebox(0,0)[lt]{\lineheight{1.25}\smash{\begin{tabular}[t]{l}$\alpha_2$\end{tabular}}}}\put(0.04677434,0.21742784){\color[rgb]{0.30588235,0.60392157,0.02352941}\makebox(0,0)[rt]{\lineheight{1.25}\smash{\begin{tabular}[t]{r}$\alpha_1$\end{tabular}}}}\put(0,0){\includegraphics[width=\unitlength,page=2]{1ext_A.pdf}}\put(0.86028114,0.06570822){\makebox(0,0)[lt]{\lineheight{1.25}\smash{\begin{tabular}[t]{l}$\alpha_2$\end{tabular}}}}\put(0.91430413,0.19928398){\makebox(0,0)[lt]{\lineheight{1.25}\smash{\begin{tabular}[t]{l}$\alpha_1 \alpha_2$\end{tabular}}}}\put(0.86018025,0.33500445){\makebox(0,0)[lt]{\lineheight{1.25}\smash{\begin{tabular}[t]{l}$\alpha_2 \alpha_1 \alpha_2$\end{tabular}}}}\put(0.58667858,0.06570822){\makebox(0,0)[rt]{\lineheight{1.25}\smash{\begin{tabular}[t]{r}$\alpha_1$\end{tabular}}}}\put(0.53368097,0.19928398){\makebox(0,0)[rt]{\lineheight{1.25}\smash{\begin{tabular}[t]{r}$\alpha_2 \alpha_1$\end{tabular}}}}\put(0.58612558,0.33500445){\makebox(0,0)[rt]{\lineheight{1.25}\smash{\begin{tabular}[t]{r}$\alpha_1 \alpha_2 \alpha_1$\end{tabular}}}}\put(0.72606164,0.00065674){\makebox(0,0)[t]{\lineheight{1.25}\smash{\begin{tabular}[t]{c}$1$\end{tabular}}}}\put(0,0){\includegraphics[width=\unitlength,page=3]{1ext_A.pdf}}\end{picture}\endgroup  		\end{scriptsize}
		\caption{}\label{fig:1ext_A}
	\end{subfigure}

		\begin{subfigure}[b]{\textwidth}
		\centering
			\begin{scriptsize}
\def\svgwidth{.6\textwidth}
		\begingroup \makeatletter \providecommand\color[2][]{\errmessage{(Inkscape) Color is used for the text in Inkscape, but the package 'color.sty' is not loaded}\renewcommand\color[2][]{}}\providecommand\transparent[1]{\errmessage{(Inkscape) Transparency is used (non-zero) for the text in Inkscape, but the package 'transparent.sty' is not loaded}\renewcommand\transparent[1]{}}\providecommand\rotatebox[2]{#2}\newcommand*\fsize{\dimexpr\f@size pt\relax}\newcommand*\lineheight[1]{\fontsize{\fsize}{#1\fsize}\selectfont}\ifx\svgwidth\undefined \setlength{\unitlength}{791.8527486bp}\ifx\svgscale\undefined \relax \else \setlength{\unitlength}{\unitlength * \real{\svgscale}}\fi \else \setlength{\unitlength}{\svgwidth}\fi \global\let\svgwidth\undefined \global\let\svgscale\undefined \makeatother \begin{picture}(1,0.37461091)\lineheight{1}\setlength\tabcolsep{0pt}\put(0,0){\includegraphics[width=\unitlength,page=1]{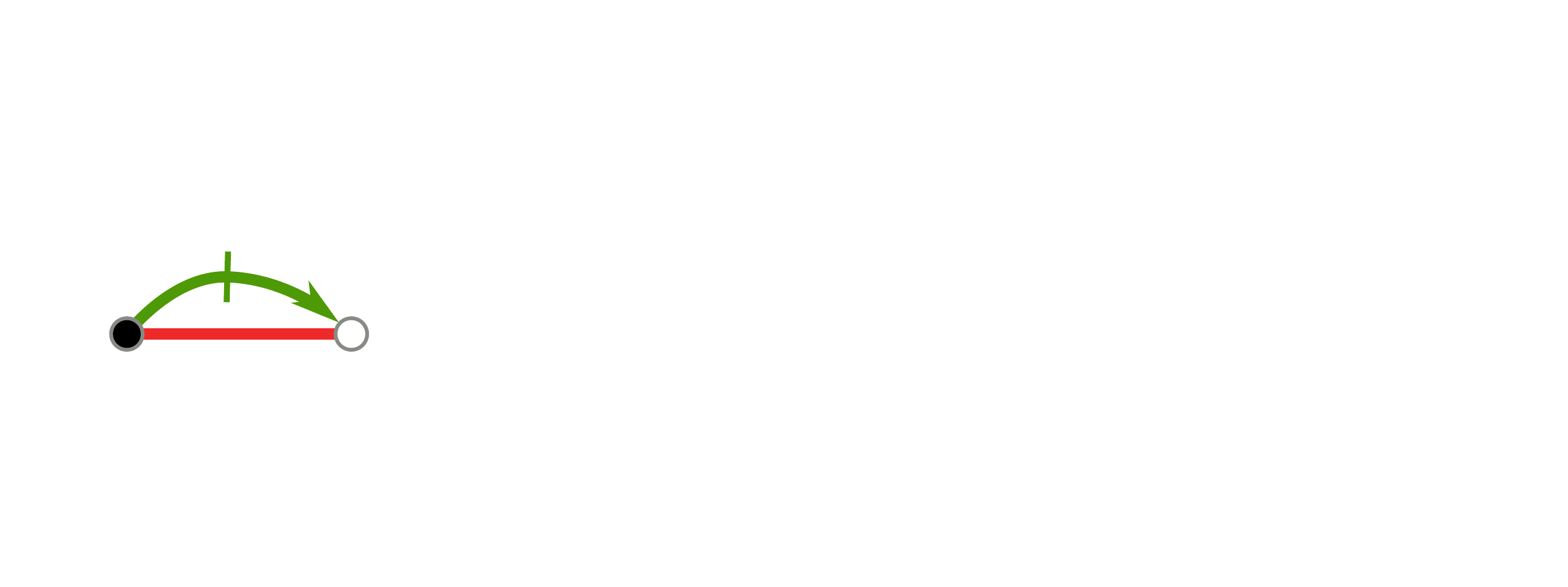}}\put(0.14556394,0.22299164){\color[rgb]{0.30588235,0.60392157,0.02352941}\makebox(0,0)[t]{\lineheight{1.25}\smash{\begin{tabular}[t]{c}$\alpha$\end{tabular}}}}\put(0,0){\includegraphics[width=\unitlength,page=2]{1ext_B.pdf}}\put(0.86028114,0.06570822){\makebox(0,0)[lt]{\lineheight{1.25}\smash{\begin{tabular}[t]{l}$\alpha^{-1}$\end{tabular}}}}\put(0.91430413,0.19928401){\makebox(0,0)[lt]{\lineheight{1.25}\smash{\begin{tabular}[t]{l}$\alpha^{-2}$\end{tabular}}}}\put(0.86018025,0.33500444){\makebox(0,0)[lt]{\lineheight{1.25}\smash{\begin{tabular}[t]{l}$\alpha^{-3}$\end{tabular}}}}\put(0.58667858,0.06570822){\makebox(0,0)[rt]{\lineheight{1.25}\smash{\begin{tabular}[t]{r}$\alpha$\end{tabular}}}}\put(0.53368097,0.19928401){\makebox(0,0)[rt]{\lineheight{1.25}\smash{\begin{tabular}[t]{r}$\alpha^2$\end{tabular}}}}\put(0.58612558,0.33500444){\makebox(0,0)[rt]{\lineheight{1.25}\smash{\begin{tabular}[t]{r}$\alpha^3$\end{tabular}}}}\put(0.72606164,0.00065671){\makebox(0,0)[t]{\lineheight{1.25}\smash{\begin{tabular}[t]{c}$1$\end{tabular}}}}\put(0,0){\includegraphics[width=\unitlength,page=3]{1ext_B.pdf}}\end{picture}\endgroup  		\end{scriptsize}
		\caption{}\label{fig:1ext_B}
	\end{subfigure}

	\caption{Two possible Cayley extensions of the $1$-maniplex. Red unmarked edges stand for $0$-adjacencies while green edges marked with a line stand for $1$-adjacencies.}
\end{figure}

        \begin{example}\label{eg:maps44}
    Let $\calK$ be a square, and let $G=\mathbb{Z}_{a}\times \mathbb{Z}_{b}$.
    Let $r_n$ act as a reflection in a vertical line for the vertical edges and a reflection in a horizontal line for the horizontal edges.
    Let $\xi(e)$ be $(-1,0)$ and $(1,0)$ for the vertical edges and $(0,-1)$ and $(0,1)$ for the horizontal edges.
    Equivalently, assign the voltage $(1,0)$ to each of the $n$-darts starting on a flag in  vertical edge on the right and the voltage $(0,1)$ to each of the $n$-darts starting on a flag in the horizontal edge on the top of the square.
    Then to build $\calK^\xi$, we take $a \times b$ copies of $\calK$, each one identified with an element of $\mathbb{Z}_a \times \mathbb{Z}_b$.
    Our definition of $r_n$ means that we glue the right edge of the square labeled $(x,y)$ to the left edge of the square labeled $(x+1,y)$, and the top edge of $(x,y)$ with the bottom edge of $(x,y+1)$.
    This gives us an $a \times b$ grid of squares glued together in a toroidal way, yielding the map $\{4,4\}_{(a,0),(0,b)}$. See Figure\nobreakspace \ref {fig:ext_44_ab}.
    \end{example}

    \begin{figure}
		\centering
			\begin{tiny}
\def\svgwidth{\textwidth}
		\begingroup \makeatletter \providecommand\color[2][]{\errmessage{(Inkscape) Color is used for the text in Inkscape, but the package 'color.sty' is not loaded}\renewcommand\color[2][]{}}\providecommand\transparent[1]{\errmessage{(Inkscape) Transparency is used (non-zero) for the text in Inkscape, but the package 'transparent.sty' is not loaded}\renewcommand\transparent[1]{}}\providecommand\rotatebox[2]{#2}\newcommand*\fsize{\dimexpr\f@size pt\relax}\newcommand*\lineheight[1]{\fontsize{\fsize}{#1\fsize}\selectfont}\ifx\svgwidth\undefined \setlength{\unitlength}{790.51042692bp}\ifx\svgscale\undefined \relax \else \setlength{\unitlength}{\unitlength * \real{\svgscale}}\fi \else \setlength{\unitlength}{\svgwidth}\fi \global\let\svgwidth\undefined \global\let\svgscale\undefined \makeatother \begin{picture}(1,0.45098141)\lineheight{1}\setlength\tabcolsep{0pt}\put(0,0){\includegraphics[width=\unitlength,page=1]{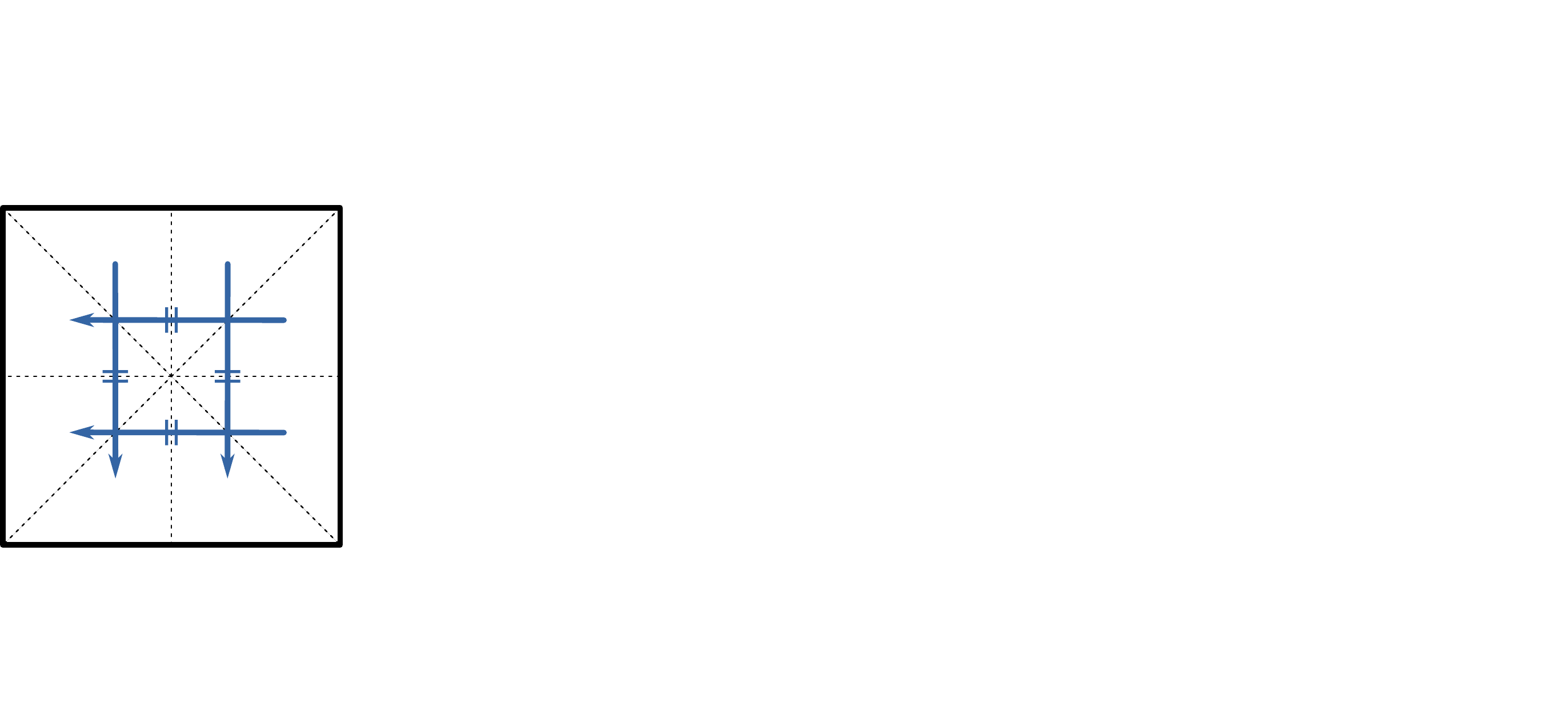}}\put(0.10926948,0.15456179){\color[rgb]{0.20392157,0.39607843,0.64313725}\makebox(0,0)[t]{\lineheight{1.25}\smash{\begin{tabular}[t]{c}$(1,0)$\end{tabular}}}}\put(0.10926948,0.26213706){\color[rgb]{0.20392157,0.39607843,0.64313725}\makebox(0,0)[t]{\lineheight{1.25}\smash{\begin{tabular}[t]{c}$(1,0)$\end{tabular}}}}\put(0,0){\includegraphics[width=\unitlength,page=2]{ext_44_ab.pdf}}\put(0.14840894,0.22627864){\color[rgb]{0.20392157,0.39607843,0.64313725}\makebox(0,0)[lt]{\lineheight{1.25}\smash{\begin{tabular}[t]{l}$(0,1)$\end{tabular}}}}\put(0.07038481,0.22627864){\color[rgb]{0.20392157,0.39607843,0.64313725}\makebox(0,0)[rt]{\lineheight{1.25}\smash{\begin{tabular}[t]{r}$(0,1)$\end{tabular}}}}\put(0,0){\includegraphics[width=\unitlength,page=3]{ext_44_ab.pdf}}\put(0.48588483,0.06744497){\color[rgb]{0,0,0}\makebox(0,0)[t]{\lineheight{1.25}\smash{\begin{tabular}[t]{c}$(0,0)$\end{tabular}}}}\put(0.5934601,0.06744497){\color[rgb]{0,0,0}\makebox(0,0)[t]{\lineheight{1.25}\smash{\begin{tabular}[t]{c}$(1,0)$\end{tabular}}}}\put(0.70103536,0.06744497){\color[rgb]{0,0,0}\makebox(0,0)[t]{\lineheight{1.25}\smash{\begin{tabular}[t]{c}$(2,0)$\end{tabular}}}}\put(0.77275215,0.06744497){\color[rgb]{0,0,0}\makebox(0,0)[t]{\lineheight{1.25}\smash{\begin{tabular}[t]{c}$(-2,0)$\end{tabular}}}}\put(0.88032742,0.06744497){\color[rgb]{0,0,0}\makebox(0,0)[t]{\lineheight{1.25}\smash{\begin{tabular}[t]{c}$(-1,0)$\end{tabular}}}}\put(0.96997353,0.01365733){\color[rgb]{0.36078431,0.20784314,0.4}\makebox(0,0)[lt]{\lineheight{1.25}\smash{\begin{tabular}[t]{l}$(a,0)$\end{tabular}}}}\put(0.88032742,0.17502023){\color[rgb]{0,0,0}\makebox(0,0)[t]{\lineheight{1.25}\smash{\begin{tabular}[t]{c}$(-1,1)$\end{tabular}}}}\put(0.48588483,0.17502023){\color[rgb]{0,0,0}\makebox(0,0)[t]{\lineheight{1.25}\smash{\begin{tabular}[t]{c}$(0,1)$\end{tabular}}}}\put(0.5934601,0.17502023){\color[rgb]{0,0,0}\makebox(0,0)[t]{\lineheight{1.25}\smash{\begin{tabular}[t]{c}$(1,1)$\end{tabular}}}}\put(0.48588483,0.24673708){\color[rgb]{0,0,0}\makebox(0,0)[t]{\lineheight{1.25}\smash{\begin{tabular}[t]{c}$(0,-2)$\end{tabular}}}}\put(0.5934601,0.24673713){\color[rgb]{0,0,0}\makebox(0,0)[t]{\lineheight{1.25}\smash{\begin{tabular}[t]{c}$(1,-2)$\end{tabular}}}}\put(0.88032742,0.24673708){\color[rgb]{0,0,0}\makebox(0,0)[t]{\lineheight{1.25}\smash{\begin{tabular}[t]{c}$(-1,-2)$\end{tabular}}}}\put(0.88032742,0.35431234){\color[rgb]{0,0,0}\makebox(0,0)[t]{\lineheight{1.25}\smash{\begin{tabular}[t]{c}$(-1,-1)$\end{tabular}}}}\put(0.77275215,0.35431234){\color[rgb]{0,0,0}\makebox(0,0)[t]{\lineheight{1.25}\smash{\begin{tabular}[t]{c}$(-2,-1)$\end{tabular}}}}\put(0.70103536,0.35431234){\color[rgb]{0,0,0}\makebox(0,0)[t]{\lineheight{1.25}\smash{\begin{tabular}[t]{c}$(2,-1)$\end{tabular}}}}\put(0.5934601,0.35431234){\color[rgb]{0,0,0}\makebox(0,0)[t]{\lineheight{1.25}\smash{\begin{tabular}[t]{c}$(1,-1)$\end{tabular}}}}\put(0.48588483,0.35431234){\color[rgb]{0,0,0}\makebox(0,0)[t]{\lineheight{1.25}\smash{\begin{tabular}[t]{c}$(0,-1)$\end{tabular}}}}\put(0.4320972,0.44395845){\color[rgb]{0.92941176,0.83137255,0}\makebox(0,0)[t]{\lineheight{1.25}\smash{\begin{tabular}[t]{c}$(0,b)$\end{tabular}}}}\put(0,0){\includegraphics[width=\unitlength,page=4]{ext_44_ab.pdf}}\end{picture}\endgroup  \end{tiny}
		\caption{The toroidal map $\left\{ 4,4 \right\}_{(a,0),(0,b)} $ as a Cayley extension. Red and green edges represent $0$- and $1$-adjacencies, respectively while blue edges marked with two lines stand for $2$-adjacencies. }\label{fig:ext_44_ab}
    \end{figure}

    \begin{example}\label{eg:toroids}
    It is not hard to see that Example\nobreakspace \ref {eg:maps44} extends to show that every $(n+1)$-toroid $\left\{ 4, 3^{n-2}, 4 \right\}_{\Lambda}$ (see \cite[Sec. 6D]{McMullenSchulte_2002_AbstractRegularPolytopes} and \cite{CollinsMontero_2021_EquivelarToroidsFew}) is a Cayley extension of the $n$-cube.
    We just need to use $\bZ^{n} / \Lambda$ as the voltage group, the reflection $R_{i}$ of the cube in direction $i$ as $r_{n+1}$ for the facets not fixed by such reflection and the projection of the vector $e_{i}$ of the standard basis of $\bZ^{n}$ into the voltage group as voltage for such facets.
    \end{example}

    A special case of Cayley extenders are those in which $r_n = Id$. We will call such an extender a \emph{canonical Cayley extender}, and will omit $r_n$ in the notation. If we say that $(\calK,\xi)$ is a canonical Cayley extender we are actually thinking about the Cayley extender $(\calK,Id,\xi)$.
    If $(\calK,\xi)$ is a canonical Cayley extender, we write $\calK^\xi$ instead of $\calK_{r_n}^\xi$ for the corresponding Cayley extension.

\begin{example} \label{eg:ditope}
	   For a maniplex $\calK$, let $G=\mathbb{Z}_2$ and let $\xi:(\calK)_{n-1}\to\mathbb{Z}_2$ be constant equal to 1. Then $\calK^\xi$ is the ditope $\{K,2\}$.
	   This is also often called the \emph{trivial extension} of $\calK$.
\end{example}

Now let us give an `external' characterization of canonical Cayley extensions.

\begin{theorem}\label{thm:charFAP}
An $(n+1)$-maniplex $\calM$ is isomorphic to $\calK^\xi$ for some canonical Cayley extender $(\calK, \xi)$ if and only if there is a group $G \leq \Gamma(\calM)$ acting regularly on the facets of $\calM$ such that for every flag $\Phi$ in $\calM$ there is an automorphism $\gamma \in G$ such that $\Phi^{n} = \Phi \gamma$.
\end{theorem}

\begin{proof}
    Proposition\nobreakspace \ref {prop:DerivedIsCayExt} establishes that $\calM$ is a Cayley extension of $\calK$ if and only if $\calM \cong \calK_{r_n}^{\xi}$ for some Cayley extender $(\calK, r_n, \xi)$. If $\calM \cong \calK^\xi$, then the voltage group $G$ acts regularly on the facets of $\calM$.
    Now, a typical flag of $\calM$ has the form $(\Psi, \alpha)$ with $\Psi$ a flag of $\calK$ and $\alpha \in G$.
    Furthermore, $(\Psi, \alpha)^{n} = (\Psi, \xi(\Phi_{n}) \alpha) = (\Psi, \alpha)(\alpha^{-1} \xi(\Phi_{n}) \alpha)$, and so $\gamma = \alpha^{-1} \xi(\Phi_{n}) \alpha$ has the desired property.

    Conversely, suppose that $G$ acts regularly on the facets of $\calM$ and has the given property.
    Then the facets of $\calM$ are all isomorphic to some maniplex $\calK$, and by definition, $\calM$ is a Cayley extension of $\calK$ by $G$.
    Again, a typical flag of $\calM$ has the form $(\Psi, \alpha)$.
    By assumption, there is $\gamma \in G$ such that $(\Psi, \alpha)^{n} = (\Psi, \alpha) \gamma = (\Psi, \alpha \gamma)$.
    Since the first coordinate of $(\Psi, \alpha)^{n}$ is $r_{n} \Psi$, then $r_{n} \Psi = \Psi$, and since $\Psi$ was arbitrary, we get $r_n = Id$, which is what we wanted to prove.

\end{proof}

In light of Theorem\nobreakspace \ref {thm:charFAP}, we will say that $\calM$ is a \emph{canonical Cayley extension (of $\calK$)} if there is a subgroup $G \leq\Gamma(\calM)$ acting regularly on the facets of $\calM$ such that for every flag $\Phi$ in $\calM$ there is an automorphism $\gamma \in G$ such that $\Phi^{n-1} = \Phi \gamma$.

\begin{remark}\label{rem:chirNotCan}
Since chiral maniplexes do not have any automorphisms that send $\Phi$ to $\Phi^{n-1}$, it follows that no chiral maniplex is a canonical Cayley extension.
However, as pointed out before, all chiral toroidal maps of type $\{4,4\}$ are (non-canonical) Cayley extensions of the square (see Example\nobreakspace \ref {eg:toroids}).
\end{remark}

If $\calM$ is regular, then we know exactly which subgroup of $\Gamma(\calM)$ we need for Theorem\nobreakspace \ref {thm:charFAP}.

\begin{proposition} \label{prop:refl-canonical}
Let $\calM$ be a regular maniplex with base flag $\Phi$ and with automorphism group $\langle \rho_0, \ldots, \rho_{n-1} \rangle$. Let $N^{+}_{n-1}$ be the normal closure of $\langle \rho_{n-1} \rangle$ in $\Gamma(\calM)$.
Then $\calM$ is a canonical Cayley extension if and only if $N^{+}_{n-1}$ acts regularly on the facets of $\calM$.
\end{proposition}

\begin{proof}
Since $\calM$ is regular, for each flag $\Psi$ there is an automorphism in $N^{+}_{n-1}$ that maps $\Psi$ to $\Psi^{n-1}$. So if $N^{+}_{n-1}$ acts regularly on the facets, then clearly $\calM$ is a canonical Cayley extension. Conversely, if $\calM$ is a canonical Cayley extension with $G$ acting regularly on the facets, then $G$ must contain $N^{+}_{n-1}$. Now, $N^{+}_{n-1}$ contains automorphisms that map each facet to each of its adjacent facets (that is, facets that share a subfacet), and by connectivity, $N^{+}_{n-1}$ acts transitively on the facets. It follows that $\Gamma = N^{+}_{n-1}$.
\end{proof}

Let us show that the notion of canonical Cayley extensions is a generalization of the Flat Amalgamation Property to general maniplexes \cite[Sec. 4E]{McMullenSchulte_2002_AbstractRegularPolytopes}.

\begin{proposition} \label{prop:fap-iff-canonical}
    Let $\calP$ be a regular polytope with facets isomorphic to $\calK$.
    Then $\calP$ satisfies the Flat Amalgamation Property with respect to its facets if and only if it is a canonical Cayley extension.
    \end{proposition}

    \begin{proof}
    Let $\gen{\rho_0,\rho_1,\ldots,\rho_{n-1}}$ be the distinguished generators of   $\aut(\cP)$ with respect to a base flag $\Phi$.
    Then $\calP$ satisfies the Flat Amalgamation Property with respect to its facets if and only if $\Gamma(\calP)=N^{+}_{n-1}\rtimes\Gamma_{n-1}$ where $N^{+}_{n-1}$ is the normal closure of the group generated by $\rho_{n-1}$, and   $\Gamma_{n-1}$ is the stabilizer of the base facet $F=\Phi_{n-1}$.

    Now, note that $\Gamma_{n-1}$ is generated by $\rho_0,\rho_1,\ldots,\rho_{n-2}$ and that $N^{+}_{n-1}$ is generated by the elements of the form $\rho_{n-1}^\alpha=\alpha^{-1}\rho_{n-1}\alpha$. Note also that $\rho_{n-1}^\alpha$ maps $\Phi\alpha$ to $(\Phi\alpha)^{n-1}$.
    In general $\aut(\cP)=N^{+}_{n-1}\Gamma_{n-1}$ and $N^{+}_{n-1}$ is normal in $\aut(\cP)$. It is also true that $N^{+}_{n-1}$ acts transitively on the facets of $\calP$. So $\aut(\cP)=N^{+}_{n-1}\rtimes\Gamma_{n-1}$ if and only if $N^{+}_{n-1}$ acts freely (and thus regularly) on the facets of $\calP$.
    The result then follows from Proposition\nobreakspace \ref {prop:refl-canonical}.
    \end{proof}

    Recall that a maniplex $\calM$ is a \emph{coextension} of $\calK$ if its dual $\calM^*$ is an extension of $\calK^*$, or in other words, if all vertex figures of $\calM$ are isomorphic to $\calK$.

	    We can define a \emph{Cayley coextender} $(\calK,r_{-1},\xi)$ analogously to a Cayley extender by replacing the facets of $\calK$ by its $0$-faces (vertices) as the domain of $\xi$, and defining that $r_{-1}$ pairs 0-faces of $\calK$ by isomorphisms between the vertex figures.
    If  $(\calK,r_{-1},\xi)$ is a Cayley coextender, then the 1-skeleton of $\calK_{r_{-1}}^\xi$ is the Cayley graph of $G$ with respect to the set $\xi(\calK_0)$ (where $\calK_0$ denotes the set of vertices of $\calK$).
    In analogy with canonical Cayley extenders, we define a \emph{canonical Cayley coextender} to be the Cayley coextender with $r_{-1} = Id$. As a mild abuse of notation, we write $\calK^\xi$ instead
    of writing $\calK_{r_{-1}}^\xi$.
Then, to know whether $\calK^\xi$ denotes an extension or a coextension of $\calK$, one only needs to know whether the domain of $\xi$ is the facets or the 0-faces of $\calK$.

\begin{example}
	    Let $G$ be a group generated by a (multi-)set $S=\{s_0,s_1,\ldots,s_{n-1}\}$ of $n$ involutions and let $\calK $ be an $n$-simplex. Number the vertices of $\calK$ using $\{0, \ldots, n-1\}$ and let $\xi$ assign $s_i$ to vertex $i$ (in this case $(\calK,\xi)$ is a canonical Cayley coextender). Then $\calK^\xi$ is the colorful polytope obtained from the Cayley graph of $\Gamma$ with respect to $S$ (see \cite{AraujoPardoHubardOliverosSchulte_2013_ColorfulPolytopesGraphs}).
	\end{example}

\subsection{Polytopality of Cayley extensions} \label{sec:polytopality}
    Of course, a natural question when building maniplexes is to determine whether or not they are (flag graphs of) abstract polytopes.
    The maniplexes that are polytopes are precisely those that satisfy the \emph{path intersection property} (see \cite{GarzaVargasHubard_2018_PolytopalityManiplexes}).
The following result is a voltage version of such result.

    \begin{theorem} \label{thm:IntProp}
(\cite[Theorem 4.2]{HubardMochan_AllPolytopesAre_preprint})
   Let $\cX$ be an $n$-premaniplex and let $\xi:\fg(\cX)\to G$ be a voltage assignment such that $\cX^\xi$ is a maniplex. Then $\cX^\xi$ is the flag graph of a polytope if and only if
  \begin{equation}\label{eq:VolPIP}
    \xi(\fg^{x,y}_{[k,n-1]}(\cX)) \cap \xi(\fg^{x,y}_{[0,m]}(\cX)) = \xi(\fg^{x,y}_{[k,m]}(\cX)),
  \end{equation}
  for all $k,m\in \{0,\ldots,n-1\}$ and all vertices $x,y$ in $\cX$.
\end{theorem}

    \begin{proposition}\label{prop:CayleyExtPoly}
	  Let $(\calK,r_n,\xi)$ be a Cayley extender. Let  $\calK_{r_n}$ be the premaniplex obtained from adding edges of color $n$ according to $r_n$. Then $\calK_{r_n}^\xi$ is a polytope if and only if the following two conditions hold: \begin{enumerate}
	    \item  $\calK$ is a polytope, and
	    \item if $\{W\in \fg^{\Phi,\Psi}_{[k,n]}(\cK_{r_{n}}) : \xi(W)=1\}$ is non-empty, then $\fg^{\Phi,\Psi}_{[k,n-1]} \left( \cK_{r_{n}} \right)$ is also non-empty.
	  \end{enumerate}

	\end{proposition}
	\begin{proof}
	Assume that $\cK_{r_{n}}^{\xi}$ is a polytope. Obviously $\calK$ is a polytope since the facets of $\calK^\xi_{r_{n}}$ are isomorphic to $\calK$.
	The second condition is a consequence of Theorem\nobreakspace \ref {thm:IntProp}.
	Indeed, since $\cK$ is connected, then there exists a path $V$ in $\cK$ connecting $\Phi$ and $\Psi$. The path $V$ can be thought as a path in $\fg_{[0,n-1]}^{\Phi,\Psi}\left( \cK_{r_{n}} \right)$.
	Observe that $\xi(V) = 1$, hence Equation\nobreakspace \textup {(\ref {eq:VolPIP})} implies that
	\[\xi\left( \fg_{[k,n-1]}^{\Phi,\Psi} \left( \cK_{r_{n}} \right) \right) = \xi\left( \fg_{[k,n]}^{\Phi,\Psi} \left( \cK_{r_{n}} \right) \right) \cap \xi\left( \fg_{[0,n-1]}^{\Phi,\Psi} \left( \cK_{r_{n}} \right) \right) \neq \emptyset,\]
	 which in turn implies that $\fg_{[k,n-1]}^{\Phi,\Psi}\left( \cK_{r_{n}} \right) $ is non-empty.

	Suppose $\calK$ is a polytope and that the second condition holds.
	Let $\Phi,\Psi$ be flags in $\calK_{r_n}$ and let $W\in\Pi^{\Phi,\Psi}_{[0,m]}\left( \cK_{r_{n}} \right)$ and $V\in \Pi^{\Phi,\Psi}_{[k,n]}\left( \cK_{r_{n}} \right)$.
	Suppose $\xi(W)=\xi(V)$.
	Without loss of generality we can assume $k>0$ and $m<n$, so this means that $\xi(W)=\xi(V)=1$.
	Because of our hypothesis, there is a path $W'\in\Pi^{\Phi,\Psi}_{[k,n-1]}\left( \cK_{r_{n}} \right)$.
	Since $\calK$ is a polytope, from $V$ and $W'$ we get a path $Z\in \Pi^{\Phi,\Psi}_{[k,m]} \left( \cK_{r_{n}} \right)$. Trivially $\xi(Z)=\xi(V)=\xi(W)=1$, so by Theorem\nobreakspace \ref {thm:IntProp}, we get that $\calK_{r_n}^\xi$ is polytopal.
	\end{proof}

	\begin{corollary}\label{coro:FAP-poly}
	Let $(\calK,\xi)$ be a canonical Cayley extender. Then $\calK^\xi$ is a polytope if and only if $\calK$ is a polytope.
	\end{corollary}
	\begin{proof}
	Use the previous theorem, considering the fact that if we delete the edges of color $n$ of a path from $\Phi$ to $\Psi$ then we obtain another such path.
	\end{proof}

Next, let us show that a group epimorphism from the voltage group of a Cayley extender naturally induces a covering of maniplexes.

	\begin{proposition}\label{prop:covers_in_Cayley}
	    Let $(\calK,r_n,\xi)$ be a Cayley extender with $\xi:(\calK)_{n-1}\to G$. Let $\pi:G \to H$ be a group epimorphism and let $\xi':(\calK)_{n-1}\to H$ be defined by $\xi'(F)=(\xi(F))\pi$.
Then $\calK_{r_n}^\xi$ covers $\calK_{r_n}^{\xi'}$.
	\end{proposition}
	\begin{proof}
	    Let us abuse notation and define $\pi:\calK_{r_n}^\xi \to \calK_{r_n}^{\xi'}$ as  $(\Phi,\gamma)\pi = (\Phi,\gamma\pi)$. To prove that this is a maniplex homomorphism, we need to show that $\pi$ commutes with each $r_i$. This is clear, since $\pi$ only multiplies the second coordinate on the right by $\pi$, whereas $r_i$ modifies the first coordinate and, if $i = n$, multiplies the second coordinate on the left.
\end{proof}

 \subsection{Flat amalgamations}\label{sec:flat}

Let us now consider the following problem, related to the ideas in \cite[Sec. 4F]{McMullenSchulte_2002_AbstractRegularPolytopes}. A maniplex is \emph{flat} if every facet is incident to every vertex. Let $\calK$ be an $n$-maniplex, and suppose that we want to construct a flat $(n+2)$-maniplex $\calM$ with medial sections (facets of vertex-figures) isomorphic to $\calK$. Then the facets of $\calM$ will be isomorphic to a coextension $\cal X$ of $\calK$ and the vertex-figures of $\calM$ will be an extension $\cal Y$ of $\calK$. 
We say that $\cal M$ is an \emph{amalgamation of $\cal X$ and $\cal Y$}. 
Given two $n+1$ maniplexes $\cal X$ and $\cal Y$, such that all the vertex figures of $\cal X$ and the facets of $\cal Y$ are isomorphic, it is natural to ask if there is an amalgamation of $\cal X$ and $\cal Y$. If we consider Cayley coextensions and extensions, then there is a natural sufficient condition for a flat amalgamation to exist:

	\begin{theorem}\label{thm:FlatAmalgamation}
	   Let $(\cal K,r_n,\xi)$ be a Cayley extender with voltage group $G$ and $(\cal K, r_{-1}, \xi')$ be a Cayley co-extender with voltage group $G'$.
	   Suppose that $(r_n r_{-1})^2$acts trivially on the flags of $\cal K$.
	   Then, there exists a flat
amalgamation of $\cal K_{r_{-1}}^{\xi'}$ and $\cal K_{r_n}^\xi$.
	   If both $\cal K_{r_n}^\xi$ and $\cal K_{r_{-1}}^{\xi'}$ (and therefore $\cal K$) are polytopal, then the amalgamation is also polytopal.
	\end{theorem}
	\begin{proof}
	Consider the premaniplex $\cal K_{r_{-1}\wedge r_{n}}$ obtained from shifting the colors of $\cal K$ up by 1, adding edges of colors 0 according to $r_{-1}$ and edges of color $n+1$ according to $r_n$. Since $(r_n r_{-1})^2$ is the identity we know that $\cal K_{r_{-1}\wedge r_n}$ is in fact a premaniplex.
	Consider the voltage group $G'\times G$ and for every flag $\Phi$ in $\cal K$ assign voltages $(\xi'(\Phi_0),1)$ to the dart of color 0 starting at $\Phi$ and $(1,\xi(\Phi_{n-1}))$ to the dart of color $n+1$ starting at $\Phi$. Assign trivial voltage to every dart with color between 1 and $n$. Let $\cal M$ be the derived graph of $\cal K_{r_{-1}\wedge r_{n}}$ with this voltage assignment.

	It is easy to see that $\cal M$ has copies of $\cal K_{r_{-1}}^{\xi'}$ as facets and copies of $\cal K_{r_n}^\xi$ as vertex figures.

	Let us show that $\cal K_{r_{-1}\wedge r_{n}}$ is a maniplex, using Proposition\nobreakspace \ref {prop:DerivedIsManiplex}. 
 Since we are assuming that $\calK_{r_{-1}}^{\xi'}$ and $\calK_{r_n}^\xi$ are maniplexes, Proposition\nobreakspace \ref {prop:DerivedIsManiplex} implies that the possible semiedges, as well as the edges of color 0 and $n+1$ that are parallel to some other edge, cannot have trivial voltage. 
	Also, since $\{\xi'(v):v\in \cal K_0\}$ is a generating set for $G'$ and $\{\xi(F):F\in (\cal K)_{n-1}\}$ is a generating set for $G$, then $\{(\xi'(v),1):v\in \calK_0\}\cup \{(1,\xi(F)):F\in (\calK)_{n-1}\}$ is a generating set for $G'\times G$. These two facts imply that $\calM$ is in fact a maniplex.

	The set of flags of the derived graph can be written as $\cal K\times G' \times G$. When we remove the edges of color $n+1$ we get one connected component for each element $\gamma\in G$, namely $\cal K \times G' \times \{\gamma\}$. Analogously, when we erase the edges of color $0$ we get components of the form $\cal K \times \{\gamma'\} \times G$. The vertex $\cal K \times G' \times \{\gamma\}$ and the facet $\cal K \times \{\gamma'\} \times G$ intersect in the medial section $\cal K \times \{\gamma'\} \times \{\gamma\}$, proving that $\cal M$ is flat.

    Since the vertex figures of $\cal M$ are isomorphic to $\cal K_{r_n}^\xi$ and the facets are isomorphic to $\cal K_{r_1}^{\xi'}$, we get that if $\cal M$ is polytopal, then $\cal K_{r_n}^\xi$ and $\cal K_{r_1}^{\xi'}$ are polytopal as well.
    Conversely, suppose that both $\cal K_{r_n}^\xi$ and $\cal K_{r_1}^{\xi'}$ are polytopal. Let $W$ be a path in $\cal K_{r_{-1}\wedge r_n}$ with colors in $[0,m]$ and $V$ be a path with colors in $[k,n+1]$ with the same end-points. We may assume without loss of generality that $k>0$ and $m<n+1$; this implies that the voltage of $W$ lies in $\Gamma'\times\{1\}$ and the voltage of $V$ lies in $\{1\}\times \Gamma$. Therefore, if $W$ and $V$ have the same voltage, they have voltage $(1,1)$.
    If we erase the edges of color $n+1$ from $\cal K_{r_{-1}\wedge r_{n}}$ we get exactly $\cal K_{r_{-1}}$, so we can apply the dual of Proposition\nobreakspace \ref {prop:CayleyExtPoly} to $W$ and get a path $W'$ with the same endpoints as $W$ but with colors in $[1,m]$.
    If instead we erase the edges of color 0 we get a shifted copy of $\cal K_{r_n}$, so by applying Proposition\nobreakspace \ref {prop:CayleyExtPoly} we get a path $V'$ with the same endpoints and colors in $[k,n]$ (remember colors are shifted by 1).
    The paths $W'$ and $V'$ have also the same voltage $(1,1)$, and since they lie in (a shifted) $\cal K$, there is a path $Z$ with colors in $[k,m]$ with the same endpoints (and also voltage $(1,1)$). Then Theorem\nobreakspace \ref {thm:IntProp} proves that the derived graph is polytopal.
	\end{proof}

	\begin{corollary}\label{coro:CanonicalFlatAmalgamation}
	    If $\cal M$ is a Cayley extension of $\cal K$ and $\cal M'$ is a Cayley coextension of $\cal K$, and one of them is canonical, then there is a flat amalgamation of $\cal M$ and $\cal M'$. 
	\end{corollary}
	\begin{proof}
	    If $\cal M$ is canonical, then $\cal M=\cal K_{r_n}^\xi$ where $r_n = Id$. Since the identity commutes with any choice for $r_{-1}$, using Theorem\nobreakspace \ref {thm:FlatAmalgamation} we are done.
	    The proof when $r_{-1}=Id$ is identical.
	\end{proof}

	Theorem\nobreakspace \ref {thm:FlatAmalgamation} gives us a way to know whether there is a flat amalgamation of a Cayley extension and a Cayley coextension of the same polytope. One only has to check if the matching $r_{-1}$ induced by the co-extension commutes with the matching $r_n$ induced by the extension.
	
	In particular, we let both $\cal M$ and $\cal M'$ be regular, canonical Cayley (co)-extensions of $\cal K$ then Corollary\nobreakspace \ref {coro:CanonicalFlatAmalgamation} is just
	\cite[Theorem 2]{Schulte_1988_AmalgamationRegularIncidence} (see also \cite[Corollary 4F11]{McMullenSchulte_2002_AbstractRegularPolytopes}).

 \section{Classic examples and universal Cayley extensions} \label{sec:universal}

    We have already seen that the ditope over $\cK$ is a Cayley extension (see Example\nobreakspace \ref {eg:ditope}). Several other extensions that have appeared in the literature are also easily described as Cayley extensions.
    In this section we will explore these examples and use them as motivation to construct several Cayley extensions with certain universal properties.
    
    \begin{example} \label{eg:flat-ext}
	   Suppose that $\calK$ is facet-bipartite. Let $G = \bD_{m} = \langle x, y : x^2 = y^2 = (xy)^m = 1 \rangle$. We can define $\xi$ as a proper $2$-coloring of the facet graph of $\calK$, using colors $x$ and $y$. Then $\calK^\xi$ is the \emph{flat extension} denoted $\calK|2m$ in \cite{Cunningham_2021_FlatExtensionsAbstract}.
	\end{example}

    \begin{example} \label{eg:2toK}
	    Let $\cK$ be a $n$-maniplex such that for every flag $\Phi$ of $\cK$ the facets of $\Phi$ and $\Phi^{n-1}$ are different. (In other words, no facet is ``glued to itself''.) Let $G=\mathbb{Z}_2^{\calK_{n-1}}$, and let $\ve_F$ denote the vector that has a $1$ in the coordinate corresponding to the facet $F$ and zeroes anywhere else. Let $\xi(F)=\ve_F$. Then $\calK^\xi$ is $\hat{2}^{\calK}$ (see Section 6 of  \cite{DouglasHubardPellicerWilson_2018_TwistOperatorManiplexes}).
	\end{example}

	\begin{example} \label{eg:colorCoded} More generally, if $c$ is a function on the facets that assigns each one a color from a set of $k$ colors, we may use voltage group $G = \bZ_2^k$ and set $\xi(F) = \ve_{c(F)}$. Then $\calK^{\xi}$ is a ``color-coded extension'' of $\cK$ as described in Section 6.1 of \cite{DouglasHubardPellicerWilson_2018_TwistOperatorManiplexes}.
	    Note that Example\nobreakspace \ref {eg:ditope} corresponds to the case where there is a single color class, and Example\nobreakspace \ref {eg:2toK} corresponds to the case where each facet is in its own color class.
	\end{example}

	\begin{example} \label{eg:2s^K-1}
	  Now let us describe another way to generalize Example\nobreakspace \ref {eg:2toK}. Let $\cK$ be a $n$-maniplex such that for every flag $\Phi$ of $\cK$ the facets of $\Phi$ and $\Phi^{n-1}$ are different. Let  $\{F_{j} : 0 \leq j < m \}$ be a labelling of the set of facets of $\calK$ so that  $F_0$ is the base facet. Choose $s \in \bN$ such that $s \geq 2$ and let \[A = \bigoplus_{0 \leq j < m}\bZ_{s} = \big\{ \vx = (x_{j})_{0\leq j < m} : x_{j} = 0 \text{ for all but finitely many } j < m \big\}.  \] Let \[U= \left\{\vx \in A : \sum_{0 \leq j < m }x_{j} \equiv 0 \pmod{s}\right\} \leq A.\]

	  Let $G = U \rtimes \langle \chi \rangle$ where $\chi: \vx \to -\vx$ for every $\vx \in U$.
	  Consider the vectors $\va_{j} = \ve_{j}-\ve_{0} \in U$, where $\ve_{i}$ denotes the vector of $A$ with $i^{th}$ entry equal to $1$ and every other entry equal to $0$.
      Note that $\va_{j} = (-1, \dots, 0, 1, 0 \dots )$ if $0<j$ and that $\va_{0}=0$.
      Observe that the set $\{(\va_j, \chi) : 0 \leq j < m\}$ is a generating set of involutions for $G$. If we let $\xi(F_j)=(\va_j,\chi)$, then $\calK^{\xi}$ is the maniplex $\hat{2}s^{\calK -1}$ (see \cite{Pellicer_2009_ExtensionsRegularPolytopes,Montero_2021_SchlaefliSymbolChiral}).
	\end{example}

\subsection{Universal extensions}

Given a pre-extender $(\cal K, r_n)$, we can complete it to make a Cayley extender whose derived graph is the freest extension of that type (meaning the freest extension with that extender as its quotient by the action of the voltage group). 

\begin{definition} \label{def:univ-ext}
The \emph{universal Cayley extension of $\calK$ with respect to $r_n$}, denoted by
$\cal U(\cal K, r_n)$, is the maniplex $\cal K_{r_n}^\xi$, where 
\[
    G(\cal K,r_n)=\gen{\{\alpha_F\}_{F\in (\cal K)_{n-1}} : \alpha_F\alpha_{r_n F} =1},
\]
and $\xi(F) = \alpha_F$.
\end{definition}

 In other words, $G(\cal K,r_n)$ is the freest group generated by one element of infinite order for each pair $\{F,r_n F\}$ with $F\neq r_n F$ and one involution for each facet fixed by $r_n$.
Using Proposition\nobreakspace \ref {prop:covers_in_Cayley}, we can see that indeed, $\cal U(\cal K,r_n)$ covers all the Cayley extensions of $\cal K$ having $(\cal K,r_n)$ as the underlying pre-extender.
Note that despite not being necessarily free, $G(\cal K,r_n)$ has the property that each element can be written uniquely as a reduced word in $\{\alpha_F\}_{F\in (\cal K)_{n-1}}$.

A particularly nice case is when $r_n = Id$, in which case we get the universal canonical Cayley extension of $\calK$, which we denote by $\calU(\calK)$.

\begin{proposition}
    If $\cal K$ is polytopal then $\cal U(\cal K, r_n)$ is polytopal.
\end{proposition}
\begin{proof}
    Let $W$ be a path in $\cal K_{r_n}$ with voltage 1 and colors in $[k,n]$.
    By Proposition\nobreakspace \ref {prop:CayleyExtPoly}, it suffices to show that there is a path (with voltage 1)
with the same starting and end point as $W$ and with colors in $[k, n-1]$. 
    We may write $W$ as $W=W_0 d_1 W_1 d_2 \ldots d_k W_k$ where $W_i$ uses colors in $[k,n-1]$ and $d_i$ has color $n$ for each $i$. Then $1=\xi(W)=\xi(d_k)\xi(d_{k-1})\ldots \xi(d_1)$.
    Because of the presentation of $G(\cK, r_{n})$ the fact that $\xi(d_k)\xi(d_{k-1})\ldots \xi(d_1)=1$ implies that there are two consecutive terms that are inverse to one another.
    Let $j$ be such that $\xi(d_j)=\xi(d_{j+1})^{-1}$. This means that $d_j$ starts at the same facet $F$ of $\calK$ as $d_{j+1}^{-1}$. Then, the endpoint $\Phi$ of $d_j$ and the starting point $\Psi$ of $d_{j+1}$ are connected by the path $W_j$ with colors in $[k,n-1]$ and also share a facet (namely, $r_n F$). But, since $\cal K$ is polytopal, this means that there is a path $V$ with colors in $[k,n-2]$ that goes from $\Phi$ to $\Psi$.
    Then, since $\restr{r_n}{F}$ is an isomorphism, the path $r_n (V)$ (the image of $V$ under this isomorphism) is a path with colors in $[k,n-2]$ that connects the starting point of $d_j$ with the endpoint of $d_{j+1}$ (see Figure\nobreakspace \ref {fig:univExt}).
    Now we consider the path  $W'=W_0d_1\ldots d_{j-1} W_{j-1} r_n(V) W_{j+1} d_{j+1}\ldots W_k$, which has the same end points  and voltage as $W$, and it uses colors in $[k,n]$, but it uses 2 darts of color $n$ less than $W$.
    If we do the same procedure to $W'$ and repeat the process $k/2$ times  (any trivial word in $G(\cK,r_{n})$ must have even length), we get a path of trivial voltage with the same end points as $W$ that does not use any dart of color $n$, proving the claim.
\end{proof}

\begin{figure}
		\centering
			\begin{scriptsize}
\def\svgwidth{.8\textwidth}
		\begingroup \makeatletter \providecommand\color[2][]{\errmessage{(Inkscape) Color is used for the text in Inkscape, but the package 'color.sty' is not loaded}\renewcommand\color[2][]{}}\providecommand\transparent[1]{\errmessage{(Inkscape) Transparency is used (non-zero) for the text in Inkscape, but the package 'transparent.sty' is not loaded}\renewcommand\transparent[1]{}}\providecommand\rotatebox[2]{#2}\newcommand*\fsize{\dimexpr\f@size pt\relax}\newcommand*\lineheight[1]{\fontsize{\fsize}{#1\fsize}\selectfont}\ifx\svgwidth\undefined \setlength{\unitlength}{856.86845434bp}\ifx\svgscale\undefined \relax \else \setlength{\unitlength}{\unitlength * \real{\svgscale}}\fi \else \setlength{\unitlength}{\svgwidth}\fi \global\let\svgwidth\undefined \global\let\svgscale\undefined \makeatother \begin{picture}(1,0.60302257)\lineheight{1}\setlength\tabcolsep{0pt}\put(0,0){\includegraphics[width=\unitlength,page=1]{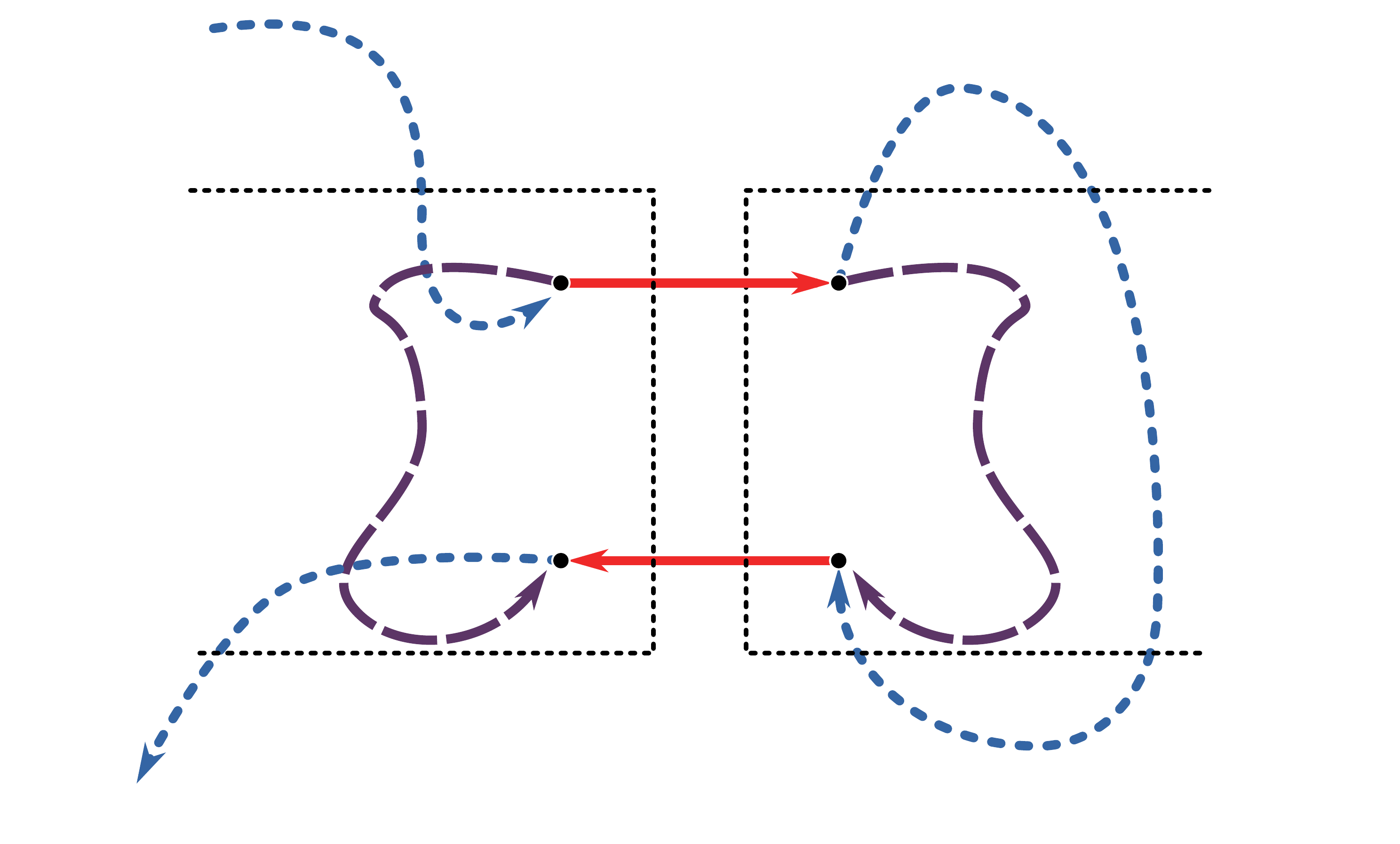}}\put(0.5992444,0.36767417){\makebox(0,0)[t]{\lineheight{1.25}\smash{\begin{tabular}[t]{c}$\Phi$\end{tabular}}}}\put(0.50000003,0.41729636){\makebox(0,0)[t]{\lineheight{1.25}\smash{\begin{tabular}[t]{c}$d_{j}$\end{tabular}}}}\put(0.46691855,0.48345927){\makebox(0,0)[rt]{\lineheight{1.24999964}\smash{\begin{tabular}[t]{r}$F$\end{tabular}}}}\put(0.11956326,0.53308149){\makebox(0,0)[rt]{\lineheight{1.24999964}\smash{\begin{tabular}[t]{r}$W_{j-1}$\end{tabular}}}}\put(0.11956326,0.11956328){\makebox(0,0)[rt]{\lineheight{1.24999964}\smash{\begin{tabular}[t]{r}$W_{j+1}$\end{tabular}}}}\put(0.84735532,0.30151126){\makebox(0,0)[lt]{\lineheight{1.24999964}\smash{\begin{tabular}[t]{l}$W_{j}$\end{tabular}}}}\put(0.5992444,0.22834608){\makebox(0,0)[t]{\lineheight{1.25}\smash{\begin{tabular}[t]{c}$\Psi$\end{tabular}}}}\put(0.53308149,0.48345927){\makebox(0,0)[lt]{\lineheight{1.24999964}\smash{\begin{tabular}[t]{l}$r_n F$\end{tabular}}}}\put(0.5,0.2188076){\makebox(0,0)[t]{\lineheight{1.25}\smash{\begin{tabular}[t]{c}$d_{j+1}$\end{tabular}}}}\put(0.68194804,0.30151126){\makebox(0,0)[rt]{\lineheight{1.24999964}\smash{\begin{tabular}[t]{r}$V$\end{tabular}}}}\put(0.31805199,0.30151126){\makebox(0,0)[lt]{\lineheight{1.24999964}\smash{\begin{tabular}[t]{l}$r_n(V)$\end{tabular}}}}\end{picture}\endgroup  \end{scriptsize}
		\caption{Polytopality of $\cU(\cK,r_{n})$}\label{fig:univExt}
\end{figure}

    Now let us see how we can adapt universal extensions to describe other classes of natural extensions. We start by considering the order of $r_{n-1} r_n$ in $\Mon(\calK^\xi)$ (we are slightly abusing notation and identifying $r_{i} \in W^{n}$ with the induced permutation of flags of $\cK^{\xi}$ ).

    \begin{proposition} \label{prop:rn-order}
    Let $(\calK, \xi)$ be a canonical Cayley extender with voltage group
    \[ G=\gen{\{\alpha_F\}_{F\in (\cal K)_{n-1}}} \]
    where $\alpha_F=\xi(F)$ for each facet $F$. Let $\Mon(\cal K^\xi) = \langle r_0, \ldots, r_n \rangle$. Then the order of $r_{n-1} r_n$ is twice the least common multiple of the orders of all $\alpha_{F'} \alpha_F$ such that $F$ shares a subfacet with $F'$.
    \end{proposition}

    \begin{proof}
    For each flag $(\Phi, x)$ we have that $(r_n r_{n-1})^2 (\Phi, x)  = (\Phi,  \alpha_{F} \alpha_{F'} x)$,
where $F$ is the facet of $\Phi$ and $F'$ is the facet of $\Phi^{n-1}$. In particular, since $\Phi$ and $\Phi^{n-1}$ have the same subfacet, $F$ and $F'$ share a subfacet, and the result follows.
    \end{proof}

    Now, the maniplex $\calU(\calK)$ has voltage group
    \[ G=\gen{\{\alpha_F\}_{F\in (\cal K)_{n-1}} : \alpha_F^2 = 1}. \]
    Then, for each pair of facets $F$ and $F'$ that share a subfacet, the element $\alpha_F \alpha_{F'}$ has infinite order, and so Proposition\nobreakspace \ref {prop:rn-order} implies that $r_{n-1} r_n$ has infinite order. Suppose that we instead wanted $r_{n-1} r_n$ to have order $2s$ for some $s \geq 2$. By Proposition\nobreakspace \ref {prop:rn-order}, the freest way to accomplish this is to add the relation $(\alpha_F \alpha_{F'})^s = 1$ for each pair of facets $F$ and $F'$ that share a subfacet.
    Let us denote the extension of $\cal K$ that uses this group by $\cal U_{2s}(\cal K)$. Note that by Corollary\nobreakspace \ref {coro:FAP-poly}, if $\cal K$ is a polytope, then so is $\cal U_{2s}(\cal K)$.

    We can obtain another, easier to describe extension using the group
    \[ G'_{2s}=\gen{\{\alpha_F\}_{F\in (\cal K)_{n-1}} : \alpha_F^2 = (\alpha_F \alpha_{F'})^s = 1}. \]
    In other words, the product $\alpha_F \alpha_{F'}$ has order $s$ regardless of whether the facets $F$ and $F'$ share a subfacet. We will denote this extension by $\cal U_{2s}'(\cal K)$. Note that $\cal U_{4}'(\cal K) \cong \hat{2}^{\cal K}$.
    
    Finally, let us reconsider the construction in Example\nobreakspace \ref {eg:2s^K-1}. Let us label the facets $\{F_j : 0 \leq j < m\}$ as in that example, and let $\alpha_i$ be $\xi(F_i)$. Then we note that $\alpha_i^2 = (\alpha_i \alpha_j)^s = 1$. Furthermore, each $\alpha_i \alpha_j$ commutes with $\alpha_k \alpha_{\ell}$. Thus, the voltage group $G$ of Example\nobreakspace \ref {eg:2s^K-1} is a quotient of
    \[ \hat{G}=\gen{\{\alpha_i\}_{i\in (\cal K)_{n-1}} : \alpha_i^2 = (\alpha_0 \alpha_i)^s = 1, [\alpha_0 \alpha_i, \alpha_0 \alpha_j]}. \]
    Now, setting $H= \gen{(\alpha_0 \alpha_i)_{i\in (\cal K)_{n-1}}}$, we find that $H \cong \bZ_s^{m-1}$. Furthermore, $\hat{G}= \gen{H, \alpha_0} = H \rtimes \gen{\alpha_0}$ of order $2s^{m-1}$. Since $G$ has the same order, $G \cong \hat{G}$. So this gives us a way to describe the voltage group abstractly.
    
    Since we have constructed several group quotients, Proposition\nobreakspace \ref {prop:covers_in_Cayley} implies the following covering relations:
    
    \[ \cal U(\cal K) \covers \calU_{2s}(\cal K) \covers \calU'_{2s}(\cal K)
    \covers \hat{2}s^{\cal K -1} \covers \cal K | 2s, \]
    where the last maniplex is only well-defined if $\cal K$ is facet-bipartite.
    Furthermore, note that for $s \geq 3$, these five constructions are distinct, though they may give the same result for some choices of $\cal K$. (For example, $\calU_{2s}$ and $\calU'_{2s}$ give the same result when applied to a simplex, where every pair of faces is adjacent.)

\begin{example}

    Let us see what we get when we apply these extensions to a square, since that is the smallest polytope such that the extensions are all distinct. 

    Let us label the facets of the square in cyclic order as $0, 1, 2, 3$. Let $\Gamma = \langle x_0, x_1, x_2, x_3 : x_i^2 = 1 \rangle$.
    Then each extension corresponds to introducing some extra relations.
    We summarise the extensions in Table\nobreakspace \ref {tab:extSquare}.

    \begin{table}[hbt]
    \begin{tabularx}{\textwidth}{|\ll{1.6}|\ll{.8}|\ll{.8}|\ll{.8}|} \hline
    \textbf{Extra relations} & \textbf{Extension} & \textbf{Size} & \textbf{Also known as} \\ \hline
    None & $\cal U(\{4\})$ & Infinite & $\{4, \infty\}$ \\ \hline
    $(x_i x_{i+1})^s = 1$ & $\cal U_{2s}(\{4\})$ & Infinite for $s \geq 2$ &  \\ \hline
    $(x_i x_j)^s = 1$ & $\cal U_{2s}'(\{4\})$ & Infinite for $s \geq 3$ & \\ \hline
    $(x_i x_j)^s = 1$, $\gen{x_0 x_1, x_0 x_2, x_0 x_3}$ is abelian & $\hat{2}s^{\{4\}-1}$ & $16s^3$ & \\ \hline
    $(x_i x_j)^2 = 1$ & $\hat{2}^{\{4\}}$ & 128 & $\{4,4\}_{(4,0)}$ \\ \hline
    $x_0 = x_2$, $x_1 = x_3$, $(x_0 x_1)^{s} = 1$ & $\{4\} | 2s$ & $16s$ & $\{4, 2s \mid 2\}$ \\ \hline
    \end{tabularx}
    \caption{Universal Cayley extensions of the square.}
    \label{tab:extSquare}
    \end{table}

    The universal extension $\cal U(\{4\})$ is the universal polyhedron of type $\{4, \infty\}$ by Corollary\nobreakspace \ref {cor:univ-reg}.

    The polyhedra $\hat{2}s^{\{4\}-1}$ are a bit obscure, but to give some idea of what we get, we used the GAP package RAMP \cite{ramp} to find that, when $s = 3$, this polyhedron is the one known as $\{4,6\}*432b$ in \cite{atlas}.
\end{example}
 \section{Symmetries of Cayley extensions} \label{sec:symmetries}

Our final goal is to determine the automorphism group and the symmetry type graph of a Cayley extension. We start with the following result which essentially states that if an automorphism projects locally, then it projects completely.

    \begin{proposition}\label{prop:projects}
        Let $\cM$ and $\cX$ be connected premaniplexes and let $p:\cal M \to \cal X$ be a covering. Let $\os{\tau}$ be an automorphism of $\cal M$. If there exists a flag $\Phi$ in $\cal M$ and an automorphism $\tau$ of $\cal X$ such that $\Phi \os{\tau} p= \Phi p \tau$, then $\os{\tau}$ projects to $\tau$, that is $\os{\tau} p = p \tau$.
    \end{proposition}
    \begin{proof}
        Let $\Psi$ be any flag of $\cal M$. Then there is a monodromy $w$ such that $\Psi=w\Phi$. Then,
        \[
        \begin{aligned}
            \Psi \os{\tau} p &= (w \Phi) (\os{\tau} p)\\
            &= w (\Phi \os{\tau} p)\\
            &= w (\Phi p \tau)\\
            &= (w \Phi) (p\tau)\\
            &= \Psi p\tau.
        \end{aligned}
        \]
\end{proof}

    \begin{corollary}\label{coro:reg_projects}
        If $p:\cal M\to \cal X$ is a covering and $\cal X$ is regular, every automorphism of $\cal M$ projects to $\cal X$.
    \end{corollary}

It is clear that in general, $\Gamma(\calK_{r_n})$ is a subgroup of $\Gamma(\calK)$; the edges that we add with label $n$ may disrupt some of the symmetry of $\calK$. Let us consider what we can say about the automorphisms of a Cayley extension when $\Gamma(\calK_{r_n}) = \Gamma(\calK)$.

  \begin{proposition}\label{prop:Cayley_projects}
        If $\Gamma(\calK_{r_n}) = \Gamma(\calK)$, then every automorphism of $\calK_{r_n}^\xi$ projects to $\calK_{r_n}$.
    \end{proposition}
    \begin{proof} 
        Let $\os{\tau}$ be an automorphism of $\calK_{r_n}^\xi$ and let $p:\calK_{r_n}^\xi\to \calK_{r_n}$ be the natural projection.
        Let $(\Phi,\alpha)$ be any flag of $\calK_{r_n}^\xi$ and let $(\Psi,\beta) = (\Phi,\alpha)\os{\tau}$. Then $(\Phi,\alpha)\os{\tau}\beta^{-1}\alpha = (\Psi,\alpha)$.
        This means that $\os{\tau}\beta^{-1}\alpha$ stabilizes the facet $(\calK, \alpha)$, which is isomorphic to $\calK$.
        Therefore, there is an automorphism $\tau$ of $\calK$ that maps $\Phi$ to $\Psi$, and by hypothesis, $\tau$ is also an automorphism of $\calK_{r_n}$. Proposition\nobreakspace \ref {prop:projects} finishes the proof.
    \end{proof}
    
    Proposition\nobreakspace \ref {prop:Cayley_projects} lets us find the whole automorphism group and symmetry type graph of certain Cayley extensions. To do so, we will use the following Theorem:
    
    \begin{theorem}[{\cite[Corollary 5.2]{Malnic_1998_GroupActionsCoverings}}]\label{thm:split_extensions}
    Let $X$ be a graph, $\xi:\Pi(X)\to G$ a voltage assignment, $A$ a subgroup of the automorphism group of $X$ and $\Omega$ a set of vertices of $X$ invariant under the action of $A$.
    If it is true that for every path $W$ with both ends on $\Omega$ with trivial voltage and every $\tau \in A$ then $\xi(W\tau)=1$, then every element of $A$ lifts to $X^\xi$ and the set of all these lifts is a semidirect product $G\rtimes A$.
    
\end{theorem}

    Recall that a polytope is \emph{hereditary} if every automorphism of each facet can be extended to an automorphism of the whole polytope that fixes that facet \cite{hereditary}.

	\begin{theorem}\label{cor:aut-gp-of-extension}
	Let $(\calK,r_n,\xi)$ be a Cayley extender.
	Suppose that $\Gamma(\cal K_{r_n})=\Gamma(\cal K)$ and that every automorphism of $\calK$ induces an automorphism of the voltage group $\Gamma$.
	Then $\calK_{r_n}^\xi$ is hereditary, $\Gamma(\calK_{r_n}^\xi) = \Gamma \rtimes \Gamma(\calK)$, and the symmetry type graph of $\calK_{r_n}^\xi$ is $\calK_{r_n}/\Gamma(\cal K)$.
\end{theorem}

	\begin{proof}
	Let $\Gamma$ be the voltage group of $(\calK,r_n,\xi)$.
	Note that our hypotheses imply that if $W$ is a path in $\calK_{r_n}$ with trivial voltage, then its image under any automorphism of $\cal K_{r_n}$ also has trivial voltage.
	Therefore, we can use Theorem\nobreakspace \ref {thm:split_extensions} taking $\Omega$ as the whole set of vertices, and we get that the automorphism group of $\cal K$ lifts as a semidirect product $\Gamma \rtimes \Gamma(\cal K)\leq \Gamma(\calK_{r_n}^\xi)$.
But, because of Proposition\nobreakspace \ref {prop:Cayley_projects} every automorphism of $\calK_{r_n}^\xi$ projects to $\cal K_{r_n}$. This proves that $\Gamma \rtimes \Gamma(\cal K)$ is the whole automorphism group of $\calK_{r_n}^\xi$.
	\end{proof}

Note that in Theorem\nobreakspace \ref {cor:aut-gp-of-extension}, the symmetry type graph that we obtain, $\calK_{r_n} / \Gamma(\calK)$, is just the symmetry type graph of $\calK$ but with additional edges of label $n$.

\begin{corollary}
If $(\calK, \xi)$ is a canonical Cayley extender and every automorphism of $\calK$ induces an automorphism of the voltage group $\Gamma$, then $\calK^\xi$ is hereditary, $\Gamma(\calK^\xi) = \Gamma \rtimes \Gamma(\calK)$, and the symmetry type graph of $\calK^\xi$ is obtained from the symmetry type graph of $\calK$ by adding semi-edges of label $n$ to every node.
\end{corollary}

\subsection{Symmetry type graphs and $r_n$-friendly sets}

Our goal now is to extend Theorem\nobreakspace \ref {cor:aut-gp-of-extension} to see what more we can say about the symmetry type graphs of Cayley extensions. In particular, we will calculate the symmetry type graph of the universal extension $\calU(\calK,r_n)$ in terms of $\calK$ and $r_n$ (see Definition\nobreakspace \ref {def:univ-ext}). 

Note that since $G$ acts regularly on the facets of $\cal K_{r_n}^\xi$ we get that $\Gamma(\cal K_{r_n}^\xi) = SG$ where $S$ is the stabilizer of $(\cal K,1)$.
Indeed, if $\tau \in \Gamma(\cal K_{r_n}^\xi)$ maps $(\Phi, 1)$ to $(\Psi, \alpha)$, then $\tau \alpha^{-1} \in S$.
However, this might not be a semidirect product, as $G$ is not necessarily normal.

It is known (see , for example, \cite[Theorem 2.3.1]{Mochan_2021_AbstractPolytopesTheir_PhDThesis}), that an automorphism of a derived graph projects if and only if it normalizes the voltage group, so $G$ is normal if and only if $S$ is the subgroup of $\Gamma(\cal K_{r_n})$ that  induces automorphisms of $G$. 

In any case, we know that $\cal K_{r_n}^\xi/G$ is isomorphic to $\cal K_{r_n}$, but the actual symmetry type graph might be a quotient of this.
However, since every flag is in the same orbit as a flag in the base facet $({\cal K},1)$, we can find out which orbits have to be identified in $\cal K_{r_n}$ to get the actual symmetry type graph just by looking at the quotient of $\cal K_{r_n}$ by the stabilizer of $({\cal K}, 1)$.
This quotient is well-defined despite the fact that the stabilizer of $(\cal K,1)$ does not act by automorphisms on $\cal K_{r_n}$. We will dig deeper into this fact after we have defined $r_n$-friendly sets.
So let us turn to the problem of determining that stabilizer.

First, let us motivate the following technical definition. Consider an automorphism $\tau$ of a Cayley extension $\cal K_{r_n}^\xi$, and fix a flag $\Phi$ of $\cal K$. Then $\tau$ takes $(\Phi, \gamma)$ to some flag $(\Phi', \gamma')$, and the flag $\Phi'$ depends not only on $\Phi$ but possibly on $\gamma$. Let us define $\tau_\gamma$ in $\Gamma(\cal K)$ by
\[ \Phi \tau_\gamma = ((\Phi, \gamma) \tau) p , \]
where $p$
is the projection onto the first coordinate.
In other words,
$\tau_\gamma$ is the automorphism of $\cal K$ induced by the restriction of $\tau$ to the facet $(\cK, \gamma)$
Thus, each automorphism $\tau$ may induce several automorphisms of $\cal K$, and we denote the set of such automorphisms by $S_{\tau}$. More generally, we may adapt the preceding process even to homomorphisms between Cayley extensions.

\begin{definition}\label{def:s_phi}
    Let $(\cal K, r_n,\xi)$ and $(\cal K, r_n',\xi')$ be Cayley extenders with voltage groups $G$ and $G'$ respectively.
    Suppose that $\cal K_{r_n}^\xi$ covers $\cal K_{r_n'}^{\xi'}$, and let $\varphi:\cal K_{r_n}^\xi\to \cal K_{r_n'}^{\xi'}$ be a covering.
    For each $\gamma\in G$ and each $\Phi\in\cal K$ define $\varphi_\gamma$ by $\Phi\varphi_\gamma = ((\Phi,\gamma)\varphi)p$, where $p$ is the projection on the first coordinate.
    Then
    \[
        S_\varphi:=\{\varphi_\gamma:\gamma\in G\}.
    \]
\end{definition}

\begin{proposition}\label{prop:FriendlyStab}
    Let $\cal K_{r_n}^\xi$ be a Cayley extension. The stabilizer of the facet $(\cal K,1)$ under the action of $\Gamma(\cal K_{r_n}^\xi)$ is isomorphic to $$\bigcup_{\tau\in \Gamma(\cal K_{r_n}^\xi)}S_\tau.$$
\end{proposition}
\begin{proof}
    Let $\Gamma=\Gamma(\cal K_{r_n}^\xi)$, let $\Gamma_{(\cal K,1)}$ be the stabilizer of the facet $(\cal K,1)$  and  $S=\bigcup_{\tau\in \Gamma}S_\tau$.
    We will define a group homomorphism $\varphi:\Gamma_{(\cal K,1)}\to S$ given by $\varphi:\tau\mapsto \tau_1$ (see Definition\nobreakspace \ref {def:s_phi}).
    In other words, we consider the action of $\tau$ on $\cal K$ by looking at its action on the base facet of $\cal K_{r_n}^\xi$.
    Observe that $\varphi$ is a homomorphism only because the domain is the stabilizer of the facet $(\cal K,1)$; in general $(\tau \alpha)_1 \neq \tau_1 \alpha_1$.

    Let us show that $\varphi$ is surjective. If $\tau_\gamma$ is an element of $S$, then there is a $\tau\in\Gamma$ such that $(\Phi,\gamma)\tau = (\Phi\tau_\gamma,\sigma)$ for every $\Phi \in \cal K$ and some fixed $\sigma \in G$.
    But $(\Phi,\gamma)$ is in the same $\Gamma$-orbit as $(\Phi,1)$, and $(\Phi\tau_\gamma,\sigma)$ is in the same $\Gamma$-orbit as $(\Phi\tau_\gamma,1)$.
    This means that there is an automorphism $\tau'\in \Gamma$ such that $(\Phi,1)\tau' = (\Phi\tau_\gamma,1)$.
    In particular $\tau'$ stabilizes $(\cal K,1)$.
    Now by definition of $\tau'_1$ we have that $(\Phi,1)\tau' = (\Phi\tau'_1,1)$, but this was $(\Phi\tau_\gamma,1)$.
    By projecting on the first coordinate and using semi-regularity we get that $\tau'_1=\tau_\gamma$, in other words $\varphi:\tau'\mapsto \tau_\gamma$.
    Note that, by definition, the image of $\varphi$ is contained in $S$, so we have proved that $S$ is the image of $\varphi$; in particular it is a group.

    Note that $\tau_1=\tau'_1$ if and only if $\tau$ and $\tau'$ act the same way on the base facet, which happens if and only if $\tau=\tau'$. This proves that $\varphi$ is an isomorphism.
\end{proof}

\begin{corollary}\label{coro:SymTypeCayExt}
    Let $(\calK,r_n,\xi)$ be a Cayley extender.
    The symmetry type graph of $\calK_{r_n}^\xi$ is the quotient of $\calK_{r_n}$ by
    $\bigcup_{\tau\in \Gamma(\cal K_{r_n}^\xi)}S_\tau.$
\end{corollary}

We now consider the problem of what possible symmetry type graphs we obtain from varying the voltage assignment on a fixed pre-extension $\calK_{r_n}$. 
To start with, we note that (having fixed a voltage assignment), each set $S_\alpha$ has the property that, if $\tau \in S_\alpha$, then for every flag $\Phi$ of $\cal K$, the flags $r_n \Phi$ and $r_n (\Phi \tau)$ are in the same orbit of $\Gamma(\cal K)$ (moreover, they differ by an element of $S_\alpha$ itself).
This ends up being the key property of these sets, and we will need a minor generalization in the following subsection, so let us introduce those ideas here.

\begin{definition}
Consider two pre-extensions $(\calK, r_n)$ and $(\calK, r_n')$.
We will say that a set $S\subset\Gamma(\calK)$ is \emph{$(r_n,r_n')$-friendly} if for every flag $\Phi$ of $\calK$ and every $\tau\in S$ there is some $\os{\tau}$ in $S$ such that $r_n'(\Phi\tau) = (r_n\Phi)\os{\tau}$. We will say that $\os{\tau}$ is an  \emph{$(r_n, r_n')$-friend} of $\tau$, or just a \emph{friend} if $r_n$ and $r_n'$ are clear from context. 
If a set is $(r_n,r_n)$-friendly we simply say that it is \emph{$r_n$-friendly.} \end{definition}

Consider two pre-extensions as in the definition. If $\Phi$ and $\Psi$ are in the same facet of $\cal K$, there is a monodromy $w$ that does not use $r_{n-1}$ and such that $\Psi=w\Phi$.
Suppose that $r_n'(\Phi\tau) = (r_n\Phi)\os{\tau}$.
Using the fact that $\tau$ and $\os{\tau}$ are automorphisms of $\cal K$ and hence they commute with $w$ and that $r_n$ and $r_n'$ also commute with $w$ (since $w$ does not use $r_{n-1}$) we get the following calculation:
\[
\begin{aligned}
        r_n'(\Psi\tau) &= r_n'((w\Phi)\tau)\\
        &= r_n' (w (\Phi\tau))\\
        &= w (r_n'(\Phi\tau))\\
        &= w ((r_n\Phi)\os{\tau})\\
        &= (w (r_n\Phi))\os{\tau}\\
        &= (r_n(w \Phi))\os{\tau}\\
        &= (r_n\Psi)\os{\tau}.\\
\end{aligned}
\]
This shows that $\os{\tau}$ does not depend on the flag $\Phi$ itself, but only on its facet on $\cal K$.
Therefore, we may write $\os{\tau}_F$ for the element in $S$ satisfying that $r_n'(\Phi\tau) = (r_n\Phi)\os{\tau}_F$ for every flag $\Phi$ in the facet $F$.

\begin{lemma}\label{lemma:r_n-friendly2}
    Let $(\cal K,r_n)$ and $(\cal K,r_n')$ be pre-extenders. Then:
    \begin{enumerate}
        \item If $S_1$ is $(r_n,r_n')$-friendly and $S_2$ is any set of isomorphisms from $\calK_{r_n}$ to $\calK_{r_n'}$, then $S_2$ and $S_1\cup S_2$ are $(r_n,r_n')$-friendly.
        \item If $S$ is $(r_n,r_n')$-friendly, then $S^{-1}$ is $(r_n',r_n)$-friendly.
        \item If $S_1$ is $(r_n,r_n')$-friendly and $S_2$ is $(r_n',r_n'')$-friendly, then $S_1S_2$ is $(r_n,r_n'')$-friendly.
    \end{enumerate}
\end{lemma}
\begin{proof}
First notice that if $\tau$ is an isomorphism from $\cal K_{r_n}$ to $\cal K_{r_n'}$ we can define $\os{\tau}_F=\tau$ for any facet $F$.
This proves that any set of such isomorphisms is $(r_n,r_n')$-friendly and that adding isomorphisms to an $(r_n,r_n')$-friendly set preserves the $(r_n,r_n')$-friendliness.

Let $S$ be $(r_n,r_n')$-friendly and let $\tau\in S$. Let $\Phi$ be a flag in a facet $F$ of $\cal K$.

Then
\[
    r_n'\Phi = r_n'(\Phi\tau^{-1}\tau) =  (r_n(\Phi\tau^{-1}))\os{\tau}_{F\tau^{-1}},
\]
Which implies that
\[
    r_n(\Phi\tau^{-1}) = (r_n'\Phi)\os{\tau}_{F\tau^{-1}}^{-1}.
\]
In other words, $S^{-1}$ is $(r_n',r_n)$-friendly and $\os{(\tau^{-1})}_F = \os{\tau}_{F\tau^{-1}}^{-1}$.

Now let $S_1$ and $S_2$ be $(r_n,r_n')$-friendly and $(r_n',r_n'')$-friendly, respectively, and let $\sigma\in S_1$ and $\tau\in S_2$.
Let $\Phi$ be a flag in a facet $F$.
If we calculate $r_n''(\Phi\sigma\tau)$ we get
\[
    r_n''(\Phi\sigma\tau) = (r_n'(\Phi\sigma))\os{\tau}_{F\sigma} = (r_n\Phi)\os{\sigma}_F\os{\tau}_{F\sigma}.
\]
This shows that $S_1 S_2$ is also $(r_n,r_n'')$-friendly (with $\os{(\sigma\tau)}_F = \os{\sigma}_F\os{\tau}_{F\sigma}$).
\end{proof}

By making $r_n=r_n'=r_n''$ we get the following lemma as a corollary.
\begin{lemma}\label{lemma:r_n-friendly}
    Let $(\cal K,r_n)$ be a pre-extender:
    \begin{enumerate}
        \item \label{item:rn-f_trivial} If $S_1$ is $r_n$-friendly and $S_2$ is any subset of $\Gamma(\calK_{r_n})$, then $S_2$ and $S_1\cup S_2$ are $r_n$-friendly.
        In particular $\{1\}$ is $r_n$-friendly.
        \item  \label{item:rn-f_inverse} If $S$ is $r_n$-friendly, then $S^{-1}$ is $r_n$-friendly.
        \item \label{item:rn-f_union} If $S_1$ and $S_2$ are $r_n$-friendly, then $S_1S_2$ is $r_n$-friendly.
    \end{enumerate}
    In particular, the group generated by the union of $r_n$-friendly sets is also $r_n$-friendly.\end{lemma}

\begin{proof}
To prove the last claim, let
\[ S = \bigcup_{i \in I} S_i \]
be an arbitrary union of $r_n$-friendly sets. If $\tau \in \langle S \rangle$, then $\tau$ has finite length and $\tau \in S_{i_1}^{\pm 1} \cdots S_{i_k}^{\pm 1}$ for some $k$. Each $S_{i_1}^{\pm 1}$ is $r_n$-friendly by Item\nobreakspace \ref {item:rn-f_inverse}, and thus this set is $r_n$-friendly by Item\nobreakspace \ref {item:rn-f_union}, meaning that $\tau$ has an $r_n$-friend in this set for each flag $\Phi$, and thus $\tau$ has an $r_n$-friend in $S$ for each $\Phi$.\end{proof}

\begin{corollary}\label{coro:r_n-friendly-grp}
    Given $(\cal K,r_n)$, there is a unique greatest $r_n$-friendly set $\frnd(\cal K,r_n)$ and it is a subgroup of $\Gamma(\cal K)$ containing $\Gamma(\cal K_{r_n})$. In particular, if $r_n = Id$, then $\frnd(\cal K, r_n) = \Gamma(\calK)$.
\end{corollary}

We will call the group $\frnd(\cal K,r_n)$ the \emph{$r_n$-friendly group of $\cal K$}. We say that an automorphism $\tau$ of $\Gamma(\cal K)$ is \emph{$r_n$-friendly} if it is an element of $\frnd(\cal K,r_n)$.

\begin{remark}
It is tempting to start by defining an automorphism $\tau$ to be $r_n$-friendly if, for every flag $\Phi$ there is an automorphism $\os{\tau}$ such that $r_n(\Phi \tau) = (r_n \Phi) \os{\tau}$. However, this misses an important nuance; namely, even if $\tau$ satisfies this condition, it need not be the case that all of its $r_n$-friends $\os{\tau}$ satisfy this condition. 
\end{remark}

We now describe a more constructive method for finding the $r_n$-friendly group of a pre-extension.

We will construct two sequences of partitions of $\cal K$, $\{\Part_k\}_{k\in \bN}$ and $\{\Part_k'\}_{k\in\bN}$ and such that
$\Part_{k+1}$ is a refinement of $\Part'_{k}$ which is itself a refinement of $\Part_{k}$ for every $k \in \bN$.
We will also use these partitions to create a decreasing sequence $\{G_k\}_{k\in\bN}$ of groups whose intersection will be $\frnd(\cal K,r_n)$.

Given a partition $\Part$ we use the expression $\Phi\sim_{\Part} \Psi$ to say that $\Phi$ and $\Psi$ are in the same element of $\Part$. 

Let us start by defining $\Part_0 = \{\cal K\}$.
We will use recursion for the rest of the construction.
If $\Part_k$ is defined, let us define
\[
    G_k=\{\tau \in \Gamma(\cal K): \Phi\sim_{\Part_k} \Phi\tau\, \  \text{for all}\ \Phi \in \cal K\}.
\]
Now define
\[
    \Part_k' = \{\Phi G_k: \Phi \in \cal K\}.
\]
Note that $\Part_k'$ refines $\Part_k$ by the definition of $G_k$.

Given two sets $Q,R\subset \cal K$, we define $A_{QR}=\{\Phi\in Q: r_n\Phi \in R\}$. Then we define:
\[
    \Part_{k+1}=\{A_{QR}:Q,R\in \Part'_k\}.
\]
In other words, $\Phi\sim_{\Part_{k+1}} \Psi$ if and only if $\Phi\sim_{\Part_k'} \Psi$ and $r_n\Phi\sim_{\Part_k'} r_n\Psi$.
Note that $\Part_{k+1}$ refines $\Part_k'$, as $A_{QR}$ is a subset of $Q$.
Since $\Part_{k+1}$ refines $\Part_k$ (by transitivity), we also know that $G_{k+1}\leq G_k$.

Let us illustrate the first few steps of the construction. As mentioned, $\Part_0 = \calK$, and so $G_0 = \Gamma(\calK)$. Then $\Part_0'$ consists of the flag-orbits of $\calK$. To obtain $\Part_1$, we essentially partition each flag-orbit; given a flag $\Phi$ in that orbit, we look at what orbit $r_n \Phi$ is in. So we have effectively labeled each flag $\Phi$ of $\calK$ by a pair consisting of its flag-orbit and the flag-orbit of $r_n \Phi$. Now $G_1$ is the subgroup of $\Gamma(\calK)$ that respects this new partition, $\Part_1'$ is the set of orbits under $G_1$, and we continue the process.

\begin{proposition}\label{prop:r_n-friendly-as-intersection}
    Let $(\cal K,r_n)$ be a pre-extender and let $\Part_k, G_k$ and $\Part'_k$ be constructed as above. Then,
    \[
        \frnd(\cal K,r_n) = \bigcap_{k=0}^\infty G_k.
    \]
\end{proposition}
\begin{proof}
    First we will prove by induction that $\frnd(\cal K,r_n)\leq G_k$.
    For $k=0$ we get that $G_0=\Gamma(\cal K)$, so there is nothing to prove.
    Now suppose that $G_k$ contains $\frnd(\cal K,r_n)$.
    Let $\Phi$ be a flag of $\cal K$ and let $\tau\in \frnd(\cal K,r_n)$.
    Then $r_n(\Phi\tau)=(r_n\Phi)\os{\tau}$ for some $\os{\tau}\in \frnd(\cal K,r_n)\leq G_k$, in other words, $r_n(\Phi\tau)\sim_{\Part_k'} r_n\Phi$.
    On the other hand $\Phi\tau \sim_{\Part_k'} \Phi$ (with $\tau$ as witness).
    Therefore, $\Phi\tau \sim_{\Part_{k+1}} \Phi$.
    But this is precisely the definition of $\tau$ being in $G_{k+1}$.
    It follows by induction that $\frnd(\cal K,r_n) \leq \bigcap_{k=0}^\infty G_k$.

    Now we will prove that $\bigcap_{k=0}^\infty G_k$ is $r_n$-friendly.
    Let $\tau \in \bigcap_{k=0}^\infty G_k$ and let $\Phi$ be any flag of $\cal K$.
    For every $k$ we know that $\tau \in G_{k+1}$, therefore $\Phi \sim _{\Part_{k+1}'} \Phi\tau$.
    Since $\Part_{k+1}'$ refines $\Part_{k+1}$ we get that $\Phi \sim_{\Part_{k+1}} \Phi \tau$.
    By definition of $\Part_{k+1}$, in particular we have that $r_n\Phi \sim_{\Part_k'} r_n(\Phi\tau)$.
    This means that there exists some $\os{\tau}_k\in G_k$ such that $r_n(\Phi\tau) = (r_n\Phi)\os{\tau}_k$.
    But, since the action of $\Gamma(\cal K)$ on $\cal K$ is free, we get that $\os{\tau}_k$ actually does not depend on $k$, so this unique $\os{\tau}_k$ is in $\bigcap_{k=0}^{\infty}G_{k}$.
    This proves that $\bigcap_{k=0}^{\infty} G_{k} $ is $r_n$-friendly, and therefore contained in $\frnd(\cal K,r_n)$.
\end{proof}

Note that in the previous proposition the common limit of the partitions $\Part_k$ and $\Part'_k$ as $k$ goes to infinity are the orbits of flags of $\cal K$ under the action of $\frnd(\cal K, r_n)$.

To prove that each automorphism of $\calK_{r_n}$ induces an $r_n$-friendly set we first prove the following more general proposition:

\begin{proposition}\label{prop:s_phi}
    Given a homomorphism $\phi:\cal K_{r_n}^\xi\to \cal K_{r_n'}^{\xi'}$,
    the set $S_\phi$ is $(r_n,r_n')$-friendly.
\end{proposition}
\begin{proof}
    We want to prove that $r_n (\Phi\phi_\gamma) = (r_n\Phi)\phi_{\gamma_1}$ for some $\gamma_1\in G$.
    Let $F$ be the facet of $\cal K$ containing $\Phi$. Then,
    \[
    \begin{aligned}
      \phi (r_n\Phi,\xi(F)\gamma) &= \phi((\Phi,\gamma)^n)\\
      &= (\phi(\Phi,\gamma))^n\\
      &= (\Phi\phi_\gamma,\gamma')^n  \quad \text { for some } \gamma'\in G'\\
      &= (r_n(\Phi\phi_\gamma),\gamma'') \quad \text { for some } \gamma''\in G'.\\
    \end{aligned}
    \]
If we apply $p$ on both sides we get $(r_n\Phi)\phi_{\xi(F)\gamma} = r_n(\Phi\phi_\gamma)$, so by making $\gamma_1 = \xi(F)\gamma$ we get our result.
\end{proof}

\begin{corollary}\label{coro:S_tau}
    Let $(\cal K, r_n, \xi)$ be a Cayley extender and let $\tau\in \Gamma(\cal K_{r_n}^\xi)$. Then the set $S_\tau$ is $r_n$-friendly.
\end{corollary}

Finally, we get Theorem\nobreakspace \ref {thm:SymTypeCayExt} as a corollary of Lemma\nobreakspace \ref {lemma:r_n-friendly} and Corollary\nobreakspace \ref {coro:S_tau}:

\begin{theorem}\label{thm:SymTypeCayExt}
    Let $(\calK,r_n,\xi)$ be a Cayley extender.
    The symmetry type graph of $\calK_{r_n}^\xi$ is a quotient of $\calK_{r_n}$ by some $r_n$-friendly group $H\leq \frnd(\calK,r_n)$.
\end{theorem}

Note that despite the fact that
$H$ might  have elements not in $\Gamma(\calK_{r_n})$
the quotient  $\calK_{r_n} / H$ is well defined because of the definition of $r_n$-friendliness.

The previous theorem tells us that by knowing the subgroups of $\frnd(\calK,r_n)$ we know all the possible symmetry type graphs for Cayley extensions of $\cal K$ having $\calK_{r_n}$ as its quotient by the Cayley group. In particular, the most regular a Cayley extension of this type can be would be one of symmetry type $\calK_{r_n}/\frnd(\calK,r_n)$.
We will see now that universal extensions have this symmetry type.

\begin{theorem}\label{thm:UnivSymType}
    The symmetry type graph of $\calU(\calK, r_n)$ is isomorphic to $\calK_{r_n} / \frnd(\cal K,r_n)$.
\end{theorem}

\begin{proof}
We know from Theorem\nobreakspace \ref {thm:SymTypeCayExt} that the symmetry type graph of $\calU(\calK, r_n)$ is $\cal K_{r_n}/\bigcup_{\tau\in \Gamma(\cal K_{r_n}^\xi)}S_\tau$.
    We also know that $\bigcup_{\tau\in \Gamma(\cal K_{r_n}^\xi)}S_\tau$ is $r_n$-friendly, and therefore contained in $\frnd(\cal K,r_n)$.
    Therefore, if we prove that $\frnd(\cal K,r_n) \subset \bigcup_{\tau\in \Gamma(\cal K_{r_n}^\xi)}S_\tau$, we get our result.

    To prove this, we will take an element $\tau \in \frnd(\cal K,r_n)$ and construct an automorphism $\varphi$ of $\cal U(\cal K, r_n)$ that acts as $\tau$ in some facet (the base facet, for simplicity). This will prove that $\tau \in S_\varphi$ and therefore our result.

    The basic idea is that we want $(\Psi, 1) \varphi = (\Psi \tau, 1)$ for every flag $\Psi$ of $\calK$. Now, we need the action of $\varphi$ to commute with the action of $r_n$. We see that
    \[ r_n ((\Psi, 1) \varphi) = r_n (\Psi \tau, 1) = (r_n (\Psi \tau), \xi(\Psi \tau)) \]
    whereas
    \[ (r_n (\Psi, 1)) \varphi = (r_n \Psi, \xi(\Psi)) \varphi. \]
    Since $\tau$ is $r_n$-friendly, that implies that $r_n(\Psi \tau) = (r_n \Psi) \os{\tau}$ for some $r_n$-friendly automorphism $\os{\tau}$. So
    we find that 
    \[ (r_n \Psi, \xi(\Psi)) \varphi = ((r_n \Psi) \os{\tau}, \xi(\Psi \tau)). \]
    Furthermore, using the fact that $\varphi$ commutes with each $r_i$ with $i < n$, we find that the action of $\varphi$ does not depend on the first coordinate; i.e., for every flag $\Lambda$ we have
    \[ (\Lambda, \xi(\Psi)) \varphi = (\Lambda \os{\tau}, \xi(\Psi \tau)). \]
    Thus, this shows that once we decide how $\varphi$ acts on flags of the form $(\Psi, 1)$, we can deduce how it must act on flags of the form $(\Psi, \gamma)$, where $\gamma$ has `length' 1 (that is, it consists of the voltage of a single flag). Using the same sort of process, we can deduce how $\varphi$ acts when $\gamma$ has length 2 and so on. Let us describe the process formally.

    Let $\Phi$ be a flag of $\cal K$ and $\Phi'=\Phi\tau$. We will construct an automorphism of $\cal U(\cal K, r_n)$ that maps $(\Phi,1)$ to $(\Phi',1)$.
    To do this, we will construct a sequence of functions $\{\varphi_k\}_{k\in \bN}$ such that:
    \begin{enumerate}
        \item\label{item:basefacet} $(\Psi,1)\varphi_k = (\Psi\tau,1)$ for every flag $\Psi$ of $\cal K$. In particular, $(\Phi,1) \varphi_k = (\Phi',1)$,
        \item\label{item:domain} $\varphi_k$ is a permutation of the set
        \[\left\{ (\Psi,\gamma) \in \cal U (\cal K, r_n) : \Psi \in \cK, \gamma \in G(\cK, r_{n}) \text{ and $\gamma$ has length $k$ or less}\right\}, \]
        where the \emph{length} of $\gamma$ is the smallest $k$ such that $\gamma$ can be written as a product of $k$ elements in $\left\{ \alpha_{F} : F \in (\cK)_{n-1} \right\} $ (see Definition\nobreakspace \ref {def:univ-ext}).
\item\label{item:sigma} For every $\gamma\in G(\cal K, r_n)$ with length $k$ or less, there is an $r_n$-friendly automorphism $\sigma_\gamma$ of $\cal K$ and an element $\gamma'\in G(\cal K,r_n)$ satisfying that $(\Psi,\gamma)\varphi_k=(\Psi\sigma_\gamma,\gamma')$
        (note that this property implies that $\varphi_k$ preserves $i$-adjacencies for $i<n$).
\item\label{item:n-adjacent} For any two $n$-adjacent flags in the domain of $\varphi_k$, their images are also $n$-adjacent.
        \item\label{item:extension} The function $\varphi_{k+1}$ is an extension of $\varphi_k$; that is, $\varphi_{k+1}$ coincides with $\varphi_k$ in the domain of $\varphi_k$. \end{enumerate}
    Once we succeed in creating this sequence, we can define $\varphi$ as its limit and it will (clearly) be an automorphism of $\cal U(\cal K,r_n)$ that acts like $\tau$ on the base facet.

First we define $\varphi_0$ by $(\Psi,1)\varphi_0=(\Psi\tau,1)$ for all flags $\Psi$ of $\calK$.

    Now suppose we already have constructed $\varphi_k$ with the desired properties. In particular, for each $\gamma\in G(\cal K,r_n)$ with length $k$ and for every flag $\Psi$ we have $(\Psi, \gamma) \varphi_k = (\Psi \sigma_{\gamma}, \gamma')$. For every $\alpha_F$ that does not cancel the first character of $\gamma$, let us define $(\Psi,\alpha_F\gamma)\varphi_{k+1}$ as
    \[
        (\Psi,\alpha_F\gamma)\varphi_{k+1} = (\Psi\os{(\sigma_\gamma\tau)}_F\os{\tau}_{F\sigma_\gamma}^{-1}, \alpha_{F\sigma_\gamma}\gamma').
    \](Recall that, if $\alpha$ is $r_n$-friendly and $\Phi$ has facet $F$, then $\os{\alpha}_F$ is the automorphism such that $r_n(\Phi \alpha) = (r_n \Phi) \os{\alpha}_F$.)  
    For the other flags in the domain of $\varphi_{k+1}$ we just make it coincide with $\varphi_k$.

    We will now prove that $\varphi_{k+1}$ satisfies the desired conditions.
    We automatically get
Items\nobreakspace \ref {item:basefacet},  \ref {item:domain} and\nobreakspace  \ref {item:extension}.
    Also, by defining $\sigma_{\alpha_F\gamma}:=\os{(\sigma_\gamma\tau)}_F\os{\tau}_{F\sigma_\gamma}^{-1}$ and $(\alpha_F\gamma)':=\alpha_{F\sigma_\gamma}\gamma'$ we get Item\nobreakspace \ref {item:sigma}.

    Let us verify that Item\nobreakspace \ref {item:n-adjacent} holds. Suppose that $(\Psi,\gamma)$ is in the domain of $\varphi_k$ and let $F$ be the facet of $\cal K$ containing $\Psi$.
    If $(\Psi,\gamma)^n$ is also in the domain of $\varphi_k$, then there is nothing to prove, since $\varphi_k$ satisfies Item\nobreakspace \ref {item:n-adjacent} and $\varphi_{k+1}$ is an extension of it.
    If $(\Psi,\gamma)^n$ is not in the domain of $\varphi_k$ we have that $(\Psi,\gamma)^n=(r_n\Psi,\alpha_F\gamma)$, and we notice that $\alpha_F\gamma$ must have length $k+1$.
    Therefore, our definition tells us that
    \[
    \begin{aligned}
        (\Psi,\gamma)^n\varphi_{k+1} &= (r_n\Psi,\alpha_F\gamma)\varphi_{k+1}\\
        &=((r_n\Psi)\os{(\sigma_\gamma\tau)}_F\os{\tau}_{F\sigma_\gamma}^{-1}, \alpha_{F\sigma_\gamma}\gamma')\\
        &=((r_n(\Psi\sigma_\gamma\tau)\os{\tau}_{F\sigma_\gamma}^{-1}, \alpha_{F\sigma_\gamma}\gamma')\\
        &=(((r_n(\Psi\sigma_\gamma))\os{\tau}_{F\sigma_\gamma}\os{\tau}_{F\sigma_\gamma}^{-1}, \alpha_{F\sigma_\gamma}\gamma')\\
        &=(r_n(\Psi\sigma_\gamma), \alpha_{F\sigma_\gamma}\gamma')\\
        &=(\Psi\sigma_\gamma,\gamma')^n\\
        &=((\Psi,\gamma)\varphi_{k+1})^n.
    \end{aligned}
    \].

    Now let us suppose that $(\Psi,\gamma)$ is in the domain of $\varphi_{k+1}$ but not in the domain of $\varphi_k$. If $(\Psi,\gamma)^n$ is also in the domain of $\varphi_{k+1}$, then, since $(\Psi,\gamma)^n=(r_n\Psi,\alpha_F\gamma)$ we must have that $\gamma$ has length $k+1$ and $\alpha_F\gamma$ has length $k$.
    We can then do a change of variable $\os{\Psi}=r_n\Psi$, $\os{\gamma}=\alpha_F\gamma$ and we are back to the previous case.
    This concludes the proof that $\varphi_{k+1}$ satisfies the desired properties.

    Then we simply define $\varphi$ as the limit of the sequence $\{\varphi_k\}_{k\in\bN}$. It is clear that $\varphi$ is an automorphism of $\cal U(\calK, r_n)$ and that it acts as $\tau$ on the base facet, so 
    $\tau \in S_{\varphi}$ and thus $\frnd(\cal K, r_n) \subseteq \bigcup_{\tau\in \Gamma(\cal K_{r_n}^\xi)}S_\tau$, which proves the theorem.

\end{proof}

\begin{corollary}
    The symmetry type graph of the universal canonical Cayley extension $\calU(\calK)$ is the symmetry type graph of $\calK$ but with semi-edges of color $n$ at every node.
\end{corollary}

As we remarked after Proposition\nobreakspace \ref {prop:r_n-friendly-as-intersection}, the limit of the partitions $\Part_k$ (and $\Part_k'$) is some partition $\Part$ corresponding to the orbits of flags of $\calK$ under the action of $\frnd(\calK, r_n)$. 
Thus it follows that the symmetry type graph of $\cal U(\cal K,r_n)$ is $\calK_{r_n}/\Part$.

\subsection{Isomorphic universal extensions}

It may happen that two universal extensions with respect to a different $r_n$ are actually the same. In what follows we will characterize when this happens.
In particular, we will show that if $\cal K$ is regular, then all the universal extensions are equal (and in fact they are equal to the free extensions described in \cite[Thm. 4D4]{McMullenSchulte_2002_AbstractRegularPolytopes}).

\begin{theorem}\label{thm:EqualUniversalExt}
   Let $\cal U(\cal K,r_n)$ and $\cal U(\cal K, r_n')$ be universal extensions of a polytope $\cal K$.
   If $\calK_{r_n} / \frnd(\cal K,r_n) = \calK_{r_n'} / \frnd(\cal K,r_n')$ then $\cal U(\cal K,r_n) \cong \cal U(\cal K, r_n')$ .
   Moreover, there is an isomorphism $\varphi:\cal U(\cal K,r_n)\to \cal U(\cal K,r_n')$ that acts as the identity on the facet $(\calK, 1)$.
\end{theorem}

\begin{proof}
    We will denote by $\alpha_F$ the generator of $G(\cal K,r_n)$ corresponding to the facet $F$ and by $\alpha_F'$ the generator of $G(\cal K,r_n')$ corresponding to $F$.

    Note that when we write $\calK_{r_n} / \frnd(\cal K,r_n) = \calK_{r_n'} / \frnd(\cal K,r_n')$ in our hypothesis, we are using the symbol for equality and not for isomorphism. This means that for every flag $\Phi$, the flag orbits $(r_n\Phi)\frnd(\cal K,r_n)$ and $(r_n'\Phi)\frnd(\cal K,r_n')$ are the same.
    Since the action of $\Gamma(\cal K)$ on the flags of $\cal K$ is free, this also implies that $\frnd(\cal K,r_n)=\frnd(\cal K,r_n')$.
    In particular, for every flag $\Phi$ in a facet $F$ of $\cal K$, there is an $r_n$-friendly automorphism $\tau_F$ depending only on $F$, satisfying that $r_n'\Phi = (r_n\Phi)\tau_F$.

    The rest of this proof will follow the same techniques as the first part of the proof for Theorem\nobreakspace \ref {thm:UnivSymType}:
We will construct a sequence of functions $\{\varphi_k\}_{k\in \bN}$ such that:
    \begin{enumerate}
        \item\label{item:domain_} The domain of $\varphi_k$ is the set of flags in $\cal U (\cal K, r_n)$ of type $(\Psi,\gamma)$ with $\Psi$ a flag of $\cal K$ and $\gamma$ an element of $G(\cal K, r_n)$ of length at most $k$.
        \item\label{item:sigma_} For every $\gamma\in G(\cal K, r_n)$ of length at most $k$, there is an automorphism $\sigma_\gamma$ in the $r_n$-friendly group of $\cal K$ and an element $\gamma'\in G(\cal K,r_n')$ satisfying that $(\Psi,\gamma) \varphi_k=(\Psi\sigma_\gamma,\gamma')$
        (note that this property implies that $\varphi_k$ preserves $i$-adjacencies for $i<n$).
        \item\label{item:n-adjacent_} If there are two $n$-adjacent flags in the domain of $\varphi_k$, their images are also $n$-adjacent.
        \item\label{item:extension_} The function $\varphi_{k+1}$ is an extension of $\varphi_k$.
    \end{enumerate}
    Once we succeed in creating this sequence, we can define $\varphi$ as its limit and it will clearly be an isomorphism between $\cal U(\cal K,r_n)$ and $\cal U(\cal K,r_n')$.

   First we define $\varphi_0$ as the identity (that is, $(\Psi,1)\varphi_0=(\Psi,1)$ for every flag $\Psi$ in $\cal K$). (Note that in principle, we could have $\varphi_0$ act on $(\Psi, 1)$ as $(\Psi \beta, 1)$ where $\beta$ is any automorphism of $\calK$, but choosing the identity simplifies matters.) Now suppose we already have constructed $\varphi_k$ with the desired properties. Now for every $\gamma\in G(\cal K,r_n)$ with length $k$ and every $\alpha_F$ that does not cancel the first character of $\gamma$, let us define:
   \[
        (\Psi,\alpha_F\gamma)\varphi_{k+1} = (\Psi\os{(\sigma_\gamma)}_F \tau_{F\sigma_\gamma}, \alpha_{F\sigma_\gamma}'\gamma').
   \]
   On the remainder of the domain of $\varphi_{k+1}$ we just make it coincide with $\varphi_k$.

   We will now prove that $\varphi_{k+1}$ satisfies the desired conditions.
   We automatically get Items\nobreakspace \ref {item:domain_} and\nobreakspace  \ref {item:extension_}. Also, by defining $\sigma_{\alpha_F\gamma}:=\os{(\sigma_\gamma)}_F \tau_{F\sigma_\gamma}$ and $(\alpha_F\gamma)':=\alpha_{F\sigma_\gamma}'\gamma'$ we get Item\nobreakspace \ref {item:sigma_}.

    Let us verify  that $\varphi_{k+1}$ satisfies Item\nobreakspace \ref {item:n-adjacent_}.
    Just as in Theorem\nobreakspace \ref {thm:UnivSymType}, we only need to prove that $((\Psi,\gamma)^n)\varphi_{k+1}=((\Psi,\gamma)\varphi_{k+1})^n$ in the case where $(\Psi,\gamma)$ is in the domain of $\varphi_k$ but $(\Psi,\gamma)^n$ is not. So we simply calculate:
    \[
    \begin{aligned}
        ((\Psi,\gamma)^n)\varphi_{k+1} &= (r_n\Psi,\alpha_F\gamma)\varphi_{k+1} \\
        &=((r_n\Psi)\os{(\sigma_\gamma)}_F \tau_{F\sigma_\gamma}, \alpha_{F\sigma_\gamma}'\gamma')\\
        &=((r_n(\Psi \sigma_\gamma)) \tau_{F\sigma_\gamma}, \alpha_{F\sigma_\gamma}'\gamma')\\
        &=(r_n'(\Psi \sigma_\gamma), \alpha_{F\sigma_\gamma}'\gamma')\\
        &= (\Psi\sigma_\gamma,\gamma')^n\\
        &= ((\Psi,\gamma)\varphi_{k+1})^n.
    \end{aligned}
   \]

    Then we simply define $\varphi$ as the limit of the sequence $\{\varphi_k\}_{k\in\bN}$ and it will clearly be an isomorphism.
\end{proof}

The converse of the previous theorem is also true. In fact:

\begin{theorem}\label{thm:EqualUnivExt2}
    Let $(\cal K,r_n)$ and $(\cal K,r_n')$ be pre-extenders.
    Then the following are equivalent:
    \begin{enumerate}
        \item\label{item:UnivIsom} There is an isomorphism $\varphi:\calU(\cal K,r_n)\to\calU(\cal K,r_n')$ that acts as the identity on the base facet $(\calK, 1)$.
        \item\label{item:friendly} There is an $(r_n,r_n')$-friendly set $S$ that includes the identity on $\cal K$.
        \item \label{item:SameSTG} $\calK_{r_n} / \frnd(\cal K,r_n) = \calK_{r_n'} / \frnd(\cal K,r_n')$.
    \end{enumerate}
\end{theorem}
\begin{proof}

\ref{item:UnivIsom}$\Rightarrow$\ref{item:friendly}: Just let $S=S_\varphi$ as in Proposition\nobreakspace \ref {prop:s_phi}.
    The fact that $\varphi$ acts as the identity on the base facet implies that $1 \in S_\varphi$.

    \ref{item:friendly}$\Rightarrow$ \ref{item:SameSTG}:
    Note that because of Lemma\nobreakspace \ref {lemma:r_n-friendly2}, the set $S\frnd(\cal K,r_n')S^{-1}$ is $r_n$-friendly, and therefore contained in $\frnd(\calK, r_n)$.
    On the other hand, since $1$ is in $S$, $S\frnd(\cal K,r_n')S^{-1}$ contains $\frnd(\cal K,r_n')$ and $S$.
    This proves that $S$ and $\frnd(\cal K, r_n')$ are contained in $\frnd(\cal K, r_n)$.

    Analogously, we may prove that $\frnd(\cal K,r_n)$ and $S$ are contained in $\frnd(\cal K,r_n')$.
    Therefore $\frnd(\cal K,r_n)=\frnd(\cal K,r_n')$.

    Now let $\Phi\in\cal K$ and let $F$ be the facet of $\cal K$ including $\Phi$.
    Note that, since $1 \in S$ which is $(r_n,r_n')$-friendly, there must be some automororphism $\os{1}_F \in S$ such that $r_n \Phi = (r_n'\Phi)\os{1}_F$, but since $S \subset \frnd(\cal K, r_n)$, we get that $r_n\Phi$ and $r_n'\Phi$ are in the same orbit under $\frnd(\cal K, r_n)$.
    This proves that $\calK_{r_n} / \frnd(\cal K,r_n) = \calK_{r_n'} / \frnd(\cal K,r_n')$.

    \ref{item:SameSTG}$\Rightarrow$\ref{item:UnivIsom}: This is Theorem\nobreakspace \ref {thm:EqualUniversalExt}.
\end{proof}

To get an even stronger version of Theorem\nobreakspace \ref {thm:EqualUnivExt2} we introduce a new concept.
Given a Cayley extender $(\cal K, r_n,\xi)$ and an automorphism $\tau$ of $\cal K$ we define
\[
    \begin{aligned}
        r_n^\tau (\Phi) &= (r_n(\Phi\tau^{-1}))\tau,\\
        \xi^\tau (F) &= \xi (F\tau^{-1}).
    \end{aligned}
\]
Of course, we may also define $r_n^\tau$ for pre-extenders.

It should be obvious that $\tau$ acts as an isomorphism from $\cal K_{r_n}$ to $\cal K_{r_n^\tau}$, and that this automorphism lifts to an isomorphism $\tau^*$ from $\cal K_{r_n}^\xi$ to $\cal K_{r_n^\tau}^{\xi^\tau}$.
In particular, $\cal U(\cal K, r_n)\cong \cal U(\cal K,r_n^\tau)$.

\begin{lemma}\label{lemma:FriendlyConjugation}
    If $S$ is $(r_n,r_n')$-friendly and $\tau \in \Gamma(\cal K)$, then
    \begin{enumerate}
        \item $\tau^{-1}S$ is $(r_n^\tau,r_n')$-friendly.
        \item $S\tau$ is $(r_n,(r_n')^\tau)$-friendly.
        \item $S^\tau$ is $(r_n^\tau,(r_n')^\tau)$-friendly.
    \end{enumerate}
\end{lemma}
\begin{proof}
    Let $\sigma\in S$ and $\Phi \in \cal K$, and let $F$ be the facet of $\Phi$.
    If we calculate $r_n'(\Phi\tau^{-1}\sigma)$ we get
    \[
    \begin{aligned}
        r_n'(\Phi\tau^{-1}\sigma) &= (r_n (\Phi\tau^{-1}))\os{\sigma}_{F\tau^{-1}}\\
        &= (r_n (\Phi\tau^{-1}))\tau\tau^{-1}\os{\sigma}_{F\tau^{-1}}\\
        &= (r_n^{\tau} \Phi)\tau^{-1}\os{\sigma}_{F\tau^{-1}}.
    \end{aligned}
    \]
    This proves that $\tau^{-1}S$ is $(r_n^\tau,r_n')$-friendly.

    If we calculate $(r_n')^\tau(\Phi\sigma\tau)$ we get
    \[
    \begin{aligned}
        (r_n')^\tau(\Phi\sigma\tau) &= (r_n'(\Phi\sigma\tau\tau^{-1}))\tau\\
        &= (r_n'(\Phi\sigma))\tau\\
        &= (r_n \Phi)\os{\sigma}_F\tau.
    \end{aligned}
    \]
    This proves that $S\tau$ is $(r_n,(r_n')^\tau)$-friendly.

    Combining these two results we get that $S^\tau$ is $(r_n^\tau,(r_n')^\tau)$-friendly.
\end{proof}

\begin{theorem}\label{thm:EqualUnivExt3}
    Let $(\cal K,r_n)$ and $(\cal K,r_n')$ be pre-extenders.
    Then the following are equivalent:
    \begin{enumerate}
        \item\label{item:UnivIsom'} There is an isomorphism $\varphi:\calU(\cal K,r_n)\to\calU(\cal K,r_n')$.
        \item\label{item:friendly'} There is a non-empty $(r_n,r_n')$-friendly set $S$.
        \item \label{item:SameSTG'} There is an automorphism $\tau$ of $\cal K$ such that $\calK_{r_n^\tau} / \frnd(\cal K,r_n)^\tau = \calK_{r_n'} / \frnd(\cal K,r_n')$.
    \end{enumerate}
\end{theorem}
\begin{proof}
    \ref{item:UnivIsom'}$\Rightarrow$\ref{item:friendly'}: Let $\tau = \varphi_1$. Recall that $\tau$ lifts as an isomorphism $\tau^*$ from $\cal U(\cal K, r_n)$ to $\cal U(\cal K,r_n^\tau)$.
    By composing $(\tau^*)^{-1}$ with $\varphi$ we get an isomorphism from $\cal U(\cal K,r_n^\tau)$ to $\cal U(\cal K,r_n')$ that acts as the identity on the base facet.Then, Theorem\nobreakspace \ref {thm:EqualUnivExt2} tells us that there is a $(r_n^\tau,r_n')$-friendly set $S$ that includes the identity.
    By Lemma\nobreakspace \ref {lemma:FriendlyConjugation} we get that $\tau S$ is $(r_n,r_n')$-friendly, and since $1\in S$, we get that $\tau S$ is non-empty.

    \ref{item:friendly'}$\Rightarrow$\ref{item:SameSTG'}: Let $\tau \in S$.
    By Lemma\nobreakspace \ref {lemma:FriendlyConjugation}, $\tau^{-1}S$ is an $(r_n^\tau,r_n')$-friendly set; and since $\tau\in S$, we get that $1\in \tau^{-1}S$.
    Then, by Theorem\nobreakspace \ref {thm:EqualUnivExt2} we get that $\calK_{r_n^\tau} / \frnd(\cal K,r_n^\tau) = \calK_{r_n'} / \frnd(\cal K,r_n')$. But using Lemma\nobreakspace \ref {lemma:FriendlyConjugation} twice, we get that $\frnd(\cal K,r_n^\tau) = \frnd(\cal K,r_n)^\tau$ (use the lemma once for each inclusion).

    \ref{item:SameSTG'}$\Rightarrow$\ref{item:UnivIsom'}:
    Theorem\nobreakspace \ref {thm:EqualUnivExt2} tells us that $\cal U(\cal K, r_n^\tau)$ is isomorphic to $\cal U(\cal K,r_n')$, but $\cal U(\cal K, r_n^\tau)$ is also isomorphic to $\cal U(\cal K, r_n)$.
\end{proof}

Now that we know when two universal Cayley extensions are isomorphic we can characterize those polytopes and maniplexes that have a unique universal extension.

\begin{theorem}\label{thm:OneUnivExt}
    All the universal extensions of $\cal K$ are isomorphic to each other, if and only if every isomorphism between two different facets of $\cal K$ as well as every involutory automorphism of a facet of $\cal K$ can be extended to an automorphism of $\cal K$.

    In this case, $\Gamma(\cal K)$ is $r_n$-friendly for every pre-extender $(\cal K,r_n)$.
\end{theorem}
\begin{proof}
    Because of Theorem\nobreakspace \ref {thm:EqualUnivExt3}, if  $\cal U(\cal K,r_n)$ is isomorphic to $\cal U(\cal K, Id)$ we get that there is a non-empty $(Id,r_n)$-friendly set $S$.
    This means that given $\sigma$ in $S$ and $\Phi$ in $\cal K$, there exists some $\os{\sigma}\in S$ such that $r_n(\Phi\sigma) = \Phi\os{\sigma}$.
    Moreover, because of statement \ref{item:SameSTG'} in Theorem\nobreakspace \ref {thm:EqualUnivExt3}, there must be some automorphism $\tau$ of $\cal K$ such that $\calK_{Id^\tau} / \frnd(\cal K,Id)^\tau = \calK_{r_n} / \frnd(\cal K,r_n)$, but $Id^\tau = Id$ for every possible $\tau$. This shows that $\calK_{Id} / \frnd(\cal K,Id) = \calK_{r_n} / \frnd(\cal K,r_n)$, and therefore using Theorem\nobreakspace \ref {thm:EqualUnivExt2} we may assume that $1$ is in $S$. 
    By taking $\sigma =  1$ we get that $\Phi$ and $r_n\Phi$ must be in the same orbit.
    In other words, $\restr{r_n}{F}$ can be extended to an automorphism of $\cal K$.

    Suppose that all the universal extensions of $\cal K$ are isomorphic.
    Let $F$ and $F'$ be two isomorphic facets of $\cal K$.
    Let $\sigma$ be an isomorphism between $F$ and $F'$.
    By our previous observation, to prove that $\sigma$ can be extended to an automorphism of $\cal K$ we only have to find a pre-extender $(\cal K,r_n)$ such that $\restr{r_n}{F} = \sigma$.
    If $F\neq F'$, we simply let $r_n$ act as $\sigma$ on $F$, as $\sigma^{-1}$ on $F'$ and as the identity in every other facet of $\cal K$.
    If $F=F'$ and $\sigma$ has order 2, we define $r_n$ as $\sigma$ on $F$ and as the identity anywhere else.

    We have proven that if all the universal extensions of $\cal K$ are isomorphic, then the isomorphisms between two different facets of $\cal K$ and the involutory automorphisms of facets of $\cal K$ can be extended to an automorphism of $\cal K$.
    Now let us assume that we can extend such automorphisms.
    Let $(\cal K,r_n)$ be an arbitrary pre-extender of $\cal K$.
    First we will prove that $\Gamma(\cal K)$ is $r_n$-friendly.

    Let $\tau\in\Gamma(\cal K)$, $\Phi$ a flag of $\cal K$, and $F$ the facet of $\Phi$.
    By hypothesis, $\sigma:=\restr{r_n}{F\tau}$, which is an isomorphism between the facets $F\tau$ and $r_n(F\tau)$, can be extended to an automorphism $\gamma\in\Gamma(\cal K)$.
    Symbolically, $r_n(\Phi\tau) = (\Phi\tau)\sigma = \Phi\tau\gamma$.
    On the other hand, $\sigma':=\restr{r_n}{r_n F}$ can also be extended to an automorphism $\gamma'\in\Gamma(\cal K)$, and this means $\Phi = r_n(r_n\Phi) = (r_n\Phi)\sigma' = (r_n\Phi)\gamma'$.
    If we now make a simple substitution we get that $r_n(\Phi\tau) = (r_n \Phi)\gamma'\tau\gamma$, and since $\tau$ and $\Phi$ were arbitrary, we proved that $\Gamma(\cal K)$ is $r_n$-friendly, or symbolically, $\frnd(\calK, r_n) = \Gamma(\cal K)$.
    This proves the second statement of our theorem.

    Finally, we see that $\cal K_{r_n}/\frnd(\cal K,r_n) = \cal K_{r_n}/\Gamma(\cal K)$, but since the restrictions of $r_n$ to a facet are isomorphisms or involutory automorphisms, our hypothesis ensures that every flag is in the same $\Gamma(\cal K)$-orbit as its $n$-adjacent flag.
    This means that $\cal K_{r_n}/\Gamma(\cal K)$ has a semi-edge at each vertex, or in other words that it is isomorphic to $\cal K_{Id}/\Gamma(\cal K)$.
    Therefore, Theorem\nobreakspace \ref {thm:EqualUnivExt3} tells us that $\cal U(\cal K,r_n)$ is isomorphic to the universal canonical Cayley extension of $\cal K$.
    This completes our proof.
\end{proof}

\begin{corollary} \label{cor:univ-reg}
    If $\calK$ is regular, then all the universal extensions of $\cal K$ are isomorphic and regular.
\end{corollary}

Every abstract regular polytope $\cal K$ has a regular extension $\{\cal K,\infty\}$, called \emph{the free extension of $\cal K$}, which covers all other regular extensions of $\cal K$ (\cite[Thm. 4D4]{McMullenSchulte_2002_AbstractRegularPolytopes}). It follows that $\{\calK,\infty\}$ covers the unique universal Cayley extension of $\cal K$, $\cal U(\cal K)$.
Also, using other results from \cite{McMullenSchulte_2002_AbstractRegularPolytopes} (for example, Lemma 4E10), we get that $\{\cal K,\infty\}$ has the Flat Amalgamation Property (i.e., it is a canonical Cayley extension of $\cal K$, see \ref{prop:fap-iff-canonical}), which means that it is covered by $\cal U(\cal K)$.
The composition of both covering homomorphisms (in any order) would be a covering homomorphism of a regular polytope to itself, but since homomorphisms are defined by the image of a single flag, this composition would have to be an automorphism.
This proves that for a regular polytope, the free extension is the same as the universal Cayley extension.

Regular maniplexes are not the only family of maniplexes with a unique universal Cayley extension, as we see in the following corollary. A $k$-maniplex with $k \geq n$ has \emph{symmetry type $2_{\overline{n-1}}$} if it has two flag orbits, and for every flag $\Phi$, the flags $\Phi$ and $r_i \Phi$ are in the same orbit if and only if $i \neq n-1$. In particular, such maniplexes have two orbits of facets.

\begin{corollary} \label{cor:univ-alt-semi-reg}
    If $\calK$ is an $n$-maniplex with symmetry type $2_{\overline{n-1}}$ and it has two isomorphism classes of facets, then all the universal extensions of $\cal K$ are isomorphic $(n+1)$-maniplexes with symmetry type $2_{\overline{n-1}}$.

\end{corollary}

\begin{proof}
    Suppose $\cal K$ has symmetry type $2_{\ol{n-1}}$ and that it has two isomorphism classes of facets. 
    Let $\sigma:F\to F'$ be any isomorphism between two facets $F$ and $F'$ of $\cal K$ (we may assume that $\sigma$ is involutory if $F=F'$, but it is not necessary).
    Since the two types of facet of $\cal K$ are non-isomorphic, $F$ and $F'$ must be in the same $\Gamma(\cal K)$-orbit, and since $\cal K$ has symmetry type $2_{\ol{n-1}}$, all the flags in the same type of facet are in the same $\Gamma(\cal K)$-orbit.
    Since $\sigma$ maps flags in $F$ to flags in the same $\Gamma(\cal K)$-orbit, it can be extended to an automorphism of $\cal K$.
    Then Theorem\nobreakspace \ref {thm:OneUnivExt} tells us that all the universal extensions of $\cal K$ are isomorphic.
\end{proof}

In the previous result, the fact that there are two isomorphism classes of facets is necessary:
    Let $\cal K$ have symmetry type $2_{\ol{n-1}}$ but with all facets isomorphic to each other.
    For example, there is a map with symmetry type $2_{\ol{2}}$ with automorphism group
    \[    \langle x, y, z : x^2 = y^2 = z^2 = (xy)^4 = (xz)^4 = (yz)^3 = xz(xy)^2 zyxy = 1 \rangle.
    \]
It is clear that $\cal K$ must have at least 2 facets $F$ and $G$ of one type.
    Let $F'$ be a facet of the other type.
    Let us define $r_n$ as some isomorphism between $F$ and $F'$, and as the identity in any other facet.

    Let $\Phi$ be a flag in $F$, and let $\tau\in \Gamma(\cal K)$ be such that $\Phi\tau$ is in $G$.
    Then $r_n(\Phi\tau) = \Phi\tau$ by construction of $r_n$.
    On the other hand, $(r_n \Phi)\tau$ is in the same orbit as $r_n\Phi$, which is the opposite orbit to $\Phi$, so there is no automorphism $\os{\tau}$ such that $(r_n\Phi)\os{\tau} = r_n(\Phi\tau)$.

\begin{corollary}\label{coro:Reg-or-Semireg}
    Let $\cal K$ be a map. Then, all the universal extensions of $\cal K$ are isomorphic to each other if and only if $\cal K$ is either regular or alternating semiregular with non isomorphic facets (i.e. it has type $2_{\ol{2}}$).
\end{corollary}
\begin{proof}
    Let $\cal K$ be a map such that all its universal extensions are isomorphic. The facets of $\cal K$ are polygons, which are regular.
    Therefore Theorem\nobreakspace \ref {thm:OneUnivExt} tells us that every flag of $\calK$ is in the same orbit as its 0-adjacent flag and its 1-adjacent flag (note that an automorphism that maps a flag to its 0-adjacent flag is an involution).
    Therefore, $\cal K$ is either regular or it has symmetry type $2_{\ol{2}}$.
\end{proof}

Note that Corollary\nobreakspace \ref {coro:Reg-or-Semireg} does not generalize to higher ranks.
For example, if $\cal M$ has no pair of facets that are isomorphic and no facet has a non-trivial automorphism, then it also has a unique universal Cayley extension.
 
\printbibliography

@Article{MalnicNedelaSkoviera_2000_LiftingGraphAutomorphisms,
  author       = {Aleksander Malni{\v{c}} and Roman Nedela and Martin {\v{S}}koviera},
  date         = {2000-10},
  journaltitle = {European Journal of Combinatorics},
  title        = {Lifting Graph Automorphisms by Voltage Assignments},
  doi          = {10.1006/eujc.2000.0390},
  number       = {7},
  pages        = {927--947},
  volume       = {21},
  publisher    = {Elsevier {BV}},
}

@Article{PellicerPotocnikToledo_2019_ExistenceResultTwo,
  author       = {Pellicer, Daniel and Poto{\v{c}}nik, Primo{\v{z}} and Toledo, Micael},
  date         = {2019-08},
  journaltitle = {Journal of Combinatorial Theory, Series A},
  title        = {An existence result on two-orbit maniplexes},
  doi          = {10.1016/j.jcta.2019.02.014},
  pages        = {226--253},
  volume       = {166},
  publisher    = {Elsevier {BV}},
}

@Article{GarzaVargasHubard_2018_PolytopalityManiplexes,
  author       = {Garza-Vargas, Jorge and Hubard, Isabel},
  date         = {2018},
  journaltitle = {Discrete Math.},
  title        = {Polytopality of maniplexes},
  doi          = {10.1016/j.disc.2018.02.017},
  issn         = {0012-365X},
  number       = {7},
  pages        = {2068--2079},
  url          = {https://doi.org/10.1016/j.disc.2018.02.017},
  volume       = {341},
  fjournal     = {Discrete Mathematics},
  mrclass      = {52B05 (05C15 05E45 52B15)},
  mrnumber     = {3802160},
  mrreviewer   = {Daniel Pellicer},
}

@Article{atlas,
  author      = {Hartley, Michael I.},
  title       = {An Atlas of Small Regular Abstract Polytopes},
  issn        = {0031-5303},
  issue       = {1},
  note        = {available online at \url{http://www.abstract-polytopes.com/atlas/}},
  pages       = {149--156},
  url         = {http://dx.doi.org/10.1007/s10998-006-0028-x},
  volume      = {53},
  affiliation = {University of Nottingham (Malaysia Campus) Jalan Broga, Semenyih, 43500, Malaysia Jalan Broga, Semenyih, 43500, Malaysia},
  journal     = {Periodica Mathematica Hungarica},
  publisher   = {Akad{\'e}miai Kiad{\'o}, co-published with Springer Science+Business Media B.V., Formerly Kluwer Academic Publishers B.V.},
  year        = {2006},
}

@Unpublished{HubardMochan_AllPolytopesAre_preprint,
  author    = {Hubard, Isabel and Mochán, Elías},
  date      = {2022},
  title     = {All polytopes are coset geometries: characterizing automorphism groups of k-orbit abstract polytopes},
  url       = {https://arxiv.org/abs/2208.00547},
  copyright = {arXiv.org perpetual, non-exclusive license},
  doi       = {10.48550/ARXIV.2208.00547},
  publisher = {arXiv},
}

@Article{Pellicer_2009_ExtensionsRegularPolytopes,
  author       = {Pellicer, Daniel},
  date         = {2009},
  journaltitle = {J. Combin. Theory Ser. A},
  title        = {Extensions of regular polytopes with preassigned {S}chl\"afli symbol},
  doi          = {10.1016/j.jcta.2008.06.004},
  issn         = {0097-3165},
  number       = {2},
  pages        = {303--313},
  url          = {http://dx.doi.org/10.1016/j.jcta.2008.06.004},
  volume       = {116},
  fjournal     = {Journal of Combinatorial Theory. Series A},
  mrclass      = {52B15 (51M20 52B11 52B20)},
  mrnumber     = {2475019},
  mrreviewer   = {Egon Schulte},
}

@InCollection{DouglasHubardPellicerWilson_2018_TwistOperatorManiplexes,
  author     = {Douglas, Ian and Hubard, Isabel and Pellicer, Daniel and Wilson, Steve},
  booktitle  = {Discrete geometry and symmetry},
  date       = {2018},
  title      = {The twist operator on maniplexes},
  doi        = {10.1007/978-3-319-78434-2_7},
  pages      = {127--145},
  publisher  = {Springer, Cham},
  series     = {Springer Proc. Math. Stat.},
  volume     = {234},
  mrclass    = {52B05},
  mrnumber   = {3816874},
  mrreviewer = {Peter A. Brooksbank},
}

@Article{sabidussi1958class,
  author    = {Sabidussi, Gert},
  title     = {On a class of fixed-point-free graphs},
  number    = {5},
  pages     = {800--804},
  volume    = {9},
  journal   = {Proceedings of the American Mathematical Society},
  publisher = {JSTOR},
  year      = {1958},
}

@InCollection{Malnic_1998_GroupActionsCoverings,
  author     = {Malnič, Aleksander},
  booktitle  = {Discrete {Mathematics}},
  date       = {1998},
  title      = {Group actions, coverings and lifts of automorphisms},
  doi        = {10.1016/S0012-365X(97)00141-6},
  number     = {1-3},
  pages      = {203--218},
  url        = {https://mathscinet.ams.org/mathscinet-getitem?mr=1603687
https://doi.org/10.1016/S0012-365X(97)00141-6},
  urldate    = {2023-03-28},
  volume     = {182},
  issn       = {0012-365X},
  mrnumber   = {1603687},
}

@Article{SchulteWeiss_1995_FreeExtensionsChiral,
  author       = {Schulte, Egon and Weiss, Asia Ivi\'c},
  date         = {1995},
  journaltitle = {Canad. J. Math.},
  title        = {Free extensions of chiral polytopes},
  doi          = {10.4153/CJM-1995-033-7},
  issn         = {0008-414X},
  number       = {3},
  pages        = {641--654},
  url          = {http://dx.doi.org/10.4153/CJM-1995-033-7},
  volume       = {47},
  fjournal     = {Canadian Journal of Mathematics. Journal Canadien de Math\'ematiques},
  mrclass      = {52B15 (52B05 52B20)},
  mrnumber     = {1346156},
  mrreviewer   = {Andrew Vince},
}

@Article{CollinsMontero_2021_EquivelarToroidsFew,
  author       = {Collins, José and Montero, Antonio},
  date         = {2021-01},
  journaltitle = {Discrete \& Computational Geometry},
  title        = {Equivelar {Toroids} with {Few} {Flag}-{Orbits}},
  doi          = {10.1007/s00454-020-00230-y},
  issn         = {1432-0444},
  language     = {en},
  number       = {2},
  pages        = {305--330},
  url          = {https://doi.org/10.1007/s00454-020-00230-y},
  urldate      = {2021-02-18},
  volume       = {65},
}

@Article{MonsonPellicerWilliams_2014_MixingMonodromyAbstract,
  author       = {Monson, B. and Pellicer, Daniel and Williams, Gordon},
  date         = {2014},
  journaltitle = {Trans. Amer. Math. Soc.},
  title        = {Mixing and monodromy of abstract polytopes},
  doi          = {10.1090/S0002-9947-2013-05954-5},
  issn         = {0002-9947},
  number       = {5},
  pages        = {2651--2681},
  url          = {http://dx.doi.org/10.1090/S0002-9947-2013-05954-5},
  volume       = {366},
  fjournal     = {Transactions of the American Mathematical Society},
  mrclass      = {52B12 (20F55 51M20)},
  mrnumber     = {3165650},
  mrreviewer   = {Maria Elisa C. Fernandes},
}

@Article{Pellicer_2012_DevelopmentsOpenProblems,
  author       = {Pellicer, Daniel},
  date         = {2012},
  journaltitle = {Ars Math. Contemp.},
  title        = {Developments and open problems on chiral polytopes},
  issn         = {1855-3966},
  number       = {2},
  pages        = {333--354},
  volume       = {5},
  fjournal     = {Ars Mathematica Contemporanea},
  mrclass      = {52B15 (05C25)},
  mrnumber     = {2929596},
  mrreviewer   = {Seyed Amin Seyed Fakhari},
}

@Article{Schulte_1988_AmalgamationRegularIncidence,
  author       = {Schulte, Egon},
  date         = {1988},
  journaltitle = {Proceedings of the London Mathematical Society. Third Series},
  title        = {Amalgamation of regular incidence-polytopes},
  doi          = {10.1112/plms/s3-56.2.303},
  issn         = {0024-6115},
  number       = {2},
  pages        = {303--328},
  volume       = {56},
  mrnumber     = {922658},
}

@Article{KwakKwon_2006_UnorientedCayleyMaps,
  author       = {Kwak, Jin Ho and Kwon, Young Soo},
  date         = {2006},
  journaltitle = {Studia Scientiarum Mathematicarum Hungarica. Combinatorics, Geometry and Topology (CoGeTo)},
  title        = {Unoriented {Cayley} maps},
  doi          = {10.1556/SScMath.43.2006.2.1},
  issn         = {0081-6906},
  number       = {2},
  pages        = {137--157},
  url          = {https://mathscinet.ams.org/mathscinet-getitem?mr=2229619},
  urldate      = {2023-05-03},
  volume       = {43},
  mrnumber     = {2229619},
}

@Article{richter2005cayley,
  author    = {Richter, R Bruce and {\v{S}}ir{\'a}{\v{n}}, Jozef and Jajcay, Robert and Tucker, Thomas W and Watkins, Mark E},
  title     = {Cayley maps},
  number    = {2},
  pages     = {189--245},
  volume    = {95},
  journal   = {Journal of Combinatorial Theory, Series B},
  publisher = {Elsevier},
  year      = {2005},
}

@Misc{ramp,
  author       = {Cunningham, G. and Mixer, M. and Williams, G.},
  title        = {{RAMP}, The Research Assistant for Maniplexes and Polytopes, {V}ersion 0.7},
  doi          = {10.5281/zenodo.6127629},
  howpublished = {\href {https://github.com/SupposeNot/RAMP} {\texttt{https://github.com/}\discretionary {}{}{}\texttt{SupposeNot/}\discretionary {}{}{}\texttt{RAMP}}},
  note         = {GAP package},
  printedkey   = {CMW},
}

@InCollection{Schulte_1985_ExtensionsRegularComplexes,
  author     = {Schulte, Egon},
  booktitle  = {Finite geometries ({W}innipeg, {M}an., 1984)},
  date       = {1985},
  title      = {Extensions of regular complexes},
  pages      = {289--305},
  publisher  = {Dekker, New York},
  series     = {Lecture Notes in Pure and Appl. Math.},
  volume     = {103},
  mrclass    = {52A25 (51M20)},
  mrnumber   = {826815},
  mrreviewer = {P. McMullen},
}

@PhdThesis{Mochan_2021_AbstractPolytopesTheir_PhDThesis,
  author      = {Mochán, Elías},
  date        = {2021},
  institution = {National University of Mexico},
  title       = {Abstract polytopes from their symmetry type graph},
  url         = {http://132.248.9.195/ptd2021/abril/0810846/Index.html},
}

@Article{CunninghamDelRioFrancosHubardToledo_2015_SymmetryTypeGraphs,
  author       = {Cunningham, Gabe and Del R\'{i}o-Francos, Mar\'{i}a and Hubard, Isabel and Toledo, Micael},
  date         = {2015},
  journaltitle = {Ann. Comb.},
  title        = {Symmetry type graphs of polytopes and maniplexes},
  doi          = {10.1007/s00026-015-0263-z},
  issn         = {0218-0006},
  number       = {2},
  pages        = {243--268},
  url          = {https://doi.org/10.1007/s00026-015-0263-z},
  volume       = {19},
  fjournal     = {Annals of Combinatorics},
  mrclass      = {52B15 (05C25 05E45)},
  mrnumber     = {3347382},
  mrreviewer   = {P. McMullen},
}

@Article{CunninghamPellicer_2014_ChiralExtensionsChiral,
  author       = {Cunningham, Gabe and Pellicer, Daniel},
  date         = {2014},
  journaltitle = {Discrete Math.},
  title        = {Chiral extensions of chiral polytopes},
  doi          = {10.1016/j.disc.2014.04.014},
  issn         = {0012-365X},
  pages        = {51--60},
  url          = {http://dx.doi.org/10.1016/j.disc.2014.04.014},
  volume       = {330},
  fjournal     = {Discrete Mathematics},
  mrclass      = {52B15 (05C90 52B05 52B12)},
  mrnumber     = {3208082},
  mrreviewer   = {Barry Monson},
}

@Article{Wilson_2012_ManiplexesPart1,
  author       = {Wilson, Steve},
  date         = {2012},
  journaltitle = {Symmetry},
  title        = {Maniplexes: {P}art 1: maps, polytopes, symmetry and operators},
  doi          = {10.3390/sym4020265},
  issn         = {2073-8994},
  number       = {2},
  pages        = {265--275},
  url          = {https://doi.org/10.3390/sym4020265},
  volume       = {4},
  fjournal     = {Symmetry},
  mrclass      = {52B05},
  mrnumber     = {2949129},
  mrreviewer   = {Seyed Amin Seyed Fakhari},
}

@Book{McMullenSchulte_2002_AbstractRegularPolytopes,
  author     = {McMullen, Peter and Schulte, Egon},
  date       = {2002},
  title      = {Abstract regular polytopes},
  doi        = {10.1017/CBO9780511546686},
  isbn       = {0-521-81496-0},
  pages      = {xiv+551},
  publisher  = {Cambridge University Press, Cambridge},
  series     = {Encyclopedia of Mathematics and its Applications},
  url        = {http://dx.doi.org/10.1017/CBO9780511546686},
  volume     = {92},
  mrclass    = {52B15 (20F55 51F15 51M20)},
  mrnumber   = {1965665 (2004a:52020)},
  mrreviewer = {Rade {\v{Z}}ivaljevi{\'c}},
}

@InCollection{SchulteWeiss_1991_ChiralPolytopes,
  author     = {Schulte, Egon and Weiss, Asia Ivi{\'c}},
  booktitle  = {Applied geometry and discrete mathematics},
  date       = {1991},
  title      = {Chiral polytopes},
  pages      = {493--516},
  publisher  = {Amer. Math. Soc., Providence, RI},
  series     = {DIMACS Ser. Discrete Math. Theoret. Comput. Sci.},
  volume     = {4},
  mrclass    = {51M20 (52B12 52C99)},
  mrnumber   = {1116373 (92f:51018)},
  mrreviewer = {P. McMullen},
}

@Article{Pellicer_2010_ExtensionsDuallyBipartite,
  author       = {Pellicer, Daniel},
  date         = {2010},
  journaltitle = {Discrete Math.},
  title        = {Extensions of dually bipartite regular polytopes},
  doi          = {10.1016/j.disc.2009.11.023},
  issn         = {0012-365X},
  number       = {12},
  pages        = {1702--1707},
  url          = {http://dx.doi.org/10.1016/j.disc.2009.11.023},
  volume       = {310},
  fjournal     = {Discrete Mathematics},
  mrclass      = {52B12 (05C10 05C25 05C76)},
  mrnumber     = {2610271},
}

@Article{AraujoPardoHubardOliverosSchulte_2013_ColorfulPolytopesGraphs,
  author       = {Araujo-Pardo, Gabriela and Hubard, Isabel and Oliveros, Deborah and Schulte, Egon},
  date         = {2013-06},
  journaltitle = {Israel Journal of Mathematics},
  title        = {Colorful polytopes and graphs},
  doi          = {10.1007/s11856-012-0136-7},
  issn         = {1565-8511},
  language     = {en},
  number       = {2},
  pages        = {647--675},
  url          = {https://doi.org/10.1007/s11856-012-0136-7},
  urldate      = {2021-10-19},
  volume       = {195},
}

@Article{CunninghamPellicer_2018_OpenProblems$k$,
  author       = {Cunningham, Gabe and Pellicer, Daniel},
  date         = {2018},
  journaltitle = {Discrete Math.},
  title        = {Open problems on {$k$}-orbit polytopes},
  doi          = {10.1016/j.disc.2018.03.004},
  issn         = {0012-365X},
  number       = {6},
  pages        = {1645--1661},
  url          = {https://doi.org/10.1016/j.disc.2018.03.004},
  volume       = {341},
  fjournal     = {Discrete Mathematics},
  mrclass      = {52B15 (05C25 05E18 52-02)},
  mrnumber     = {3784786},
  mrreviewer   = {Gordon Ian Williams},
}

@Article{Montero_2021_SchlaefliSymbolChiral,
  author       = {Montero, Antonio},
  date         = {2021-11},
  journaltitle = {Discrete Mathematics},
  title        = {On the {Schläfli} symbol of chiral extensions of polytopes},
  doi          = {10.1016/j.disc.2021.112507},
  issn         = {0012-365X},
  language     = {en},
  number       = {11},
  pages        = {112507},
  url          = {https://www.sciencedirect.com/science/article/pii/S0012365X2100220X},
  urldate      = {2021-09-17},
  volume       = {344},
}

@InCollection{hereditary,
  author     = {Mixer, Mark and Schulte, Egon and Weiss, Asia Ivi\'{c}},
  booktitle  = {Rigidity and symmetry},
  title      = {Hereditary polytopes},
  doi        = {10.1007/978-1-4939-0781-6_14},
  pages      = {279--302},
  publisher  = {Springer, New York},
  series     = {Fields Inst. Commun.},
  url        = {https://doi.org/10.1007/978-1-4939-0781-6_14},
  volume     = {70},
  mrclass    = {52B15},
  mrnumber   = {3329279},
  mrreviewer = {P. McMullen},
  year       = {2014},
}

@Article{Cunningham_2021_FlatExtensionsAbstract,
  author       = {Cunningham, Gabe},
  date         = {2021},
  journaltitle = {The Art of Discrete and Applied Mathematics},
  title        = {Flat extensions of abstract polytopes},
  doi          = {10.26493/2590-9770.1351.f7e},
  number       = {3},
  pages        = {Paper No. 3.06, 14},
  volume       = {4},
  mrnumber     = {4312690},
}

\end{document}